\documentclass[onefignum,onetabnum]{siamart171218}
\usepackage{lipsum}
\usepackage{amsfonts}

  \usepackage{caption}
    \usepackage{subcaption}
    \usepackage{graphicx}

\usepackage{epstopdf}
\usepackage{algorithmic}
\ifpdf
  \DeclareGraphicsExtensions{.eps,.pdf,.png,.jpg}
\else
  \DeclareGraphicsExtensions{.eps}
\fi

\newsiamremark{remark}{Remark}
\newsiamremark{assumption}{Assumption}
\crefname{hypothesis}{Hypothesis}{Hypotheses}
\newsiamthm{claim}{Claim}

\headers{Distributed learning with compressed gradients}{S. Khirirat, H. R. Feyzmahdavian, and M. Johansson}

\title{Distributed learning with compressed gradients\thanks{Submitted to the editors DATE. Work in progress.
\funding{This work was partially supported by the Wallenberg AI, Autonomous Systems and Software Program (WASP) funded by the Knut and Alice Wallenberg Foundation.}}}

\author{Sarit Khirirat\thanks{ Automatic Control Department, Royal Institute of Technology (KTH), Stockholm, Sweden
  (\email{sarit@kth.se,mikaelj@kth.se}).}
\and Hamid~Reza~Feyzmahdavian\thanks{ABB Corporate Research Center, Västerås, Sweden 
  (\email{hamid.feyzmahdavian@se.abb.com}).}
\and  Mikael~Johansson, \footnotemark[2]}

\usepackage{amsopn}
\DeclareMathOperator{\diag}{diag}

\begin{document}

\maketitle

% REQUIRED
\begin{abstract}
Asynchronous computation and gradient compression have emerged as two key techniques for achieving scalability in distributed optimization for large-scale machine learning. This paper presents a unified analysis framework for distributed gradient methods operating with staled and compressed gradients. Non-asymptotic bounds on convergence rates and information exchange are derived for several optimization algorithms. These bounds give explicit expressions for step-sizes and characterize how the amount of asynchrony and the compression accuracy affect iteration and communication complexity guarantees. Numerical results highlight convergence properties of different gradient compression algorithms and confirm that fast convergence under limited information exchange is indeed possible.    
\end{abstract}

% REQUIRED

% REQUIRED

\section{Introduction}
  Several problems in machine learning involve empirical risk minimization and can be cast as separable optimization problems
    \begin{equation}
    \mathop {\rm minimize}\limits_{x\in\mathbb{R}^d} f(x) = \sum_{i=1}^m f_i(x).
    \label{eqn:Problem}    
    \end{equation}
    Here, each component function $f_i:\mathbb{R}^d\rightarrow\mathbb{R}$ represents the loss for a single data point or subset of data points, and is assumed to be smooth and have a Lipschitz-continuous gradient.     
    The standard first-order method for solving \eqref{eqn:Problem} is \emph{gradient descent} (GD)
    \begin{equation}
    x_{k+1} = x_k - \gamma_k \sum_{i=1}^m \nabla f_i( x_{k} )
    \label{eqn:GD}  
    \end{equation}
	for some positive step-size $\gamma_k$. 
	However, when the number of component functions  $m$ is extremely large, the computation cost per iteration of GD becomes significant, and one typically resorts to \emph{stochastic gradient descent} (where the gradient is evaluated at a single randomly chosen data point in every iteration) or leverages on data-parallelism by distributing the gradient computations on multiple parallel machines; see, \emph{e.g.}~\cite{shamir2016without,SGDNeedell2014,zhang2014asynchronous,li2013parameter}. 
	%When the component functions $f_i$ are convex, the method is guaranteed to converge toward the %optimum at rate $\mathcal{O}(1/k)$. If, in addition, $f(x)$ is strongly convex, then the method has %linear convergence $\mathcal{O}(\rho^k)$. In absence of convexity, only convergence to stationary %points can be guaranteed unless additional assumptions are %imposed~\cite{polyak1987introduction,nesterov2013introductory}. 
	The latter option leads to \emph{master-server} architectures, where a central master node maintains the current parameter iterate and workers evaluate gradients of the loss on individual subsets of the global data. There are both \emph{synchronous} and \emph{asynchronous} versions of this master-worker architecture.

	     In the \emph{synchronous} master-worker architecture, the master node waits for all the gradients computed by the workers before it makes an update~\cite{zhang2014asynchronous,chen2016revisiting}. Insisting on a synchronous operation leads to long communication times (waiting for the slowest worker to complete) and the benefits of parallelization diminish as the number of workers increases. \emph{Asynchronous} master-worker architectures, such as  parameter server~\cite{li2013parameter}, attempt to alleviate this bottleneck by letting the master update its parameters every time it receives new information from a worker. Since the workers now operate on inconsistent data, the training accuracy may degrade and there is a risk that the optimization process diverges. 
	    	
	    The natural implementation of distributed gradient descent in the parameter server framework is referred to as~\emph{incremental aggregate gradient} (IAG)~\cite{blatt2007convergent}. Given an initial point $x_0$ and a step-size $\gamma$, the master executes the updates
	    % initial 
	    %  where the update rather uses the staled (or delayed) gradients, \cite{blatt2007convergent}.   %The IAG update can be summarized as follows: given the initial point $x_0$ and the fixed, %constant step size $\gamma$   
	    
	    \begin{equation}
	    x_{k+1} = x_k - \gamma \sum_{i=1}^m\nabla f_i( x_{k-\tau_k^i} ). 
	    \label{eqn:IAG}    
	    \end{equation}
		Here $\tau_k^i$ describes the staleness of the gradient information from worker $i$ available to the master at iteration $k$. Under the assumption of bounded staleness,  
	    $\tau_k^i \leq \tau$ for all $k,i$, convergence guarantees for IAG have been established for several classes of loss functions, see \emph{e.g.}  \cite{blatt2007convergent,gurbuzbalaban2017convergence,tseng2014incrementally,aytekin2016analysis}. 
	    
	     A drawback with the master-worker architecture is the massive amount of data exchanged between  workers and master. This is especially true when the parameter dimension $d$ is large and we try to scale up the number of worker machines $m$. 
	    Recently, several authors have proposed various gradient compression algorithms for reducing the network cost in distributed machine learning~\cite{wangni2017gradient,alistarh2016qsgd,de2015taming,wen2017terngrad,magnusson2017convergence}. The compression algorithms can be both randomized~\cite{wangni2017gradient,alistarh2016qsgd,wen2017terngrad} and deterministic~\cite{alistarh2016qsgd,magnusson2017convergence}, and empirical studies have demonstrated that they can yield significant savings in network traffic~\cite{wangni2017gradient,alistarh2016qsgd,wen2017terngrad}. However, the vast majority of the work on gradient compression do not provide convergence guarantees, and the few convergence results that exist often make restrictive assumptions, \emph{e.g.} that component function gradients are uniformly bounded in norm. Even though this assumption is valid for a certain classes of optimization problems, it is always violated when the objective function is strongly convex \cite{nguyen2018sgd}. In addition, the theoretical support for quantifying the trade-off between iteration and communication complexity is limited, and there are very few general results which allow to characterize the impact of different compression strategies on the convergence rate guarantees.

    \textbf{Contributions.} We establish a unified framework for both synchronous and asynchronous distributed optimization using compressed gradients. The framework builds on unbiased randomized quantizers (URQs), a class of gradient compression schemes which cover the ones proposed in \cite{wangni2017gradient,alistarh2016qsgd}. We establish per-iteration convergence rate guarantees for both GD and IAG with URQ compression. The convergence rate guarantees give explicit formulas for how quantization accuracy and staleness bounds affect the expected time to reach an $\varepsilon$-optimal solution. These results allows us to characterize the trade-off between iteration and communication complexity under gradient compression. Finally, we validate the theoretical results on large-scale parameter estimation problems.

    \textbf{Related work.} Although the initial results on communication complexity of convex optimization appeared over 30 years ago~\cite{tsitsiklis1987communication}, the area has attracted strong renewed interest due to the veritable explosion of data and parameter sizes in deep learning. Several heuristic gradient compression techniques have been proposed and evaluated empirically~\cite{wangni2017gradient,wen2017terngrad,1-bit-stochastic-gradient-descents}. Most compression schemes are based on sparsification~\cite{wangni2017gradient}, quantization~\cite{wen2017terngrad,magnusson2017convergence}, or combinations of the two~\cite{alistarh2016qsgd}; they are either randomized~\cite{wangni2017gradient,wen2017terngrad} or deterministic~\cite{alistarh2016qsgd}. While the majority of papers on gradient compression have a practical focus, several recent works establish theoretical convergence guarantees for gradient compression. In some cases, convergence guarantees are asymptotic, while other papers provide non-asymptotic bounds. The work which is most closely related to the present paper is~\cite{alistarh2016qsgd} and~\cite{de2015taming}. In particular~\cite{alistarh2016qsgd} proposes a low-precision quantizer and derives non-asymptotic convergence guarantees for (synchronous) stochastic gradient descent, while~\cite{de2015taming} introduces an analysis framework based on rate-supermartingales and develops probabilistic guarantees for quantized SGD.

    %==================================================================================
    %
    % Notations and assumptions 
    %
    %==================================================================================
    \section{Notations and  Assumptions}
    We let $\mathbb{N}, \mathbb{N}_0$ be a set of natural numbers and of natural numbers including zero. For any integers $a,b$ with $a\leq b$, $[a,b] = \{a,a+1,\ldots, b-1 ,b \}$. For a vector $x\in\mathbb{R}^d$, $x^i$ denotes its $i^{\rm th}$ element, ${\rm sign}(x^i)$ the sign of its $i^{\rm th}$ element, and ${\rm sign}(x)$ is its sign vector;  $\|x\|_0$ denotes the $\ell_0$ norm of $x$ or the number of its non-zero elements, $\|x\|$ is its Euclidean norm, and ${\rm supp}(x)$ is its support set, \emph{i.e.} 
    \begin{align*}
        {\rm supp}(x) = \{ i \ | \ x^i \neq 0 \}.
    \end{align*}

    In addition, we impose the following typical assumptions on Problem \eqref{eqn:Problem}.
    \begin{assumption}
    \label{assum:FiLipschitz}
    Each $f_i:\mathbb{R}^d \rightarrow \mathbb{R}$ is convex and has Lipschitz continuous gradient with $L$, i.e $ \forall x,y\in \mathbb{R}^d$
    \begin{align*}
     f_i(x) + \langle \nabla f_i(x),y-x \rangle   \leq  f_i(y) \leq f_i(x) + \langle \nabla f_i(x),y-x\rangle + \frac{L}{2}\| y-x\|^2.
    \end{align*}
    \end{assumption}
    Note that Assumption \ref{assum:FiLipschitz} implies that $f$ also has Lipschitz continuous gradient with $\bar L \leq m L$.

 \begin{assumption}
    \label{assum:mustronglyconvex}
    The function $f:\mathbb{R}^d\rightarrow\mathbb{R}$ is $\mu-$strongly convex, \emph{i.e.} there exits $\mu>0$ such that
    \begin{align*}
        f(y) \geq f(x) + \langle \nabla f(x),y-x \rangle + \frac{\mu}{2}\| y-x \|^2 && \forall x,y \in \mathbb{R}^d.
    \end{align*}
    \end{assumption}
    
    \iffalse
    \begin{assumption}
    \label{assum:ConvexFcn}
    The function $f:\mathbb{R}^d\rightarrow\mathbb{R}$ is  convex, i.e. 
    \begin{align*}
        f(y) \geq f(x) + \langle \nabla f(x),y-x \rangle && \forall x,y \in \mathbb{R}^d.
    \end{align*}
    \end{assumption}
\fi

If the data sparsity pattern is exploited, then one can often derive a smaller Lipschitz constant, which allows larger step-sizes and faster algorithm convergence. Although it is difficicult to quantify the gradient sparsity for general loss functions, it is possible to do so under the following additional assumptions.
    \begin{assumption}\label{assum:SparsityEmpiricalLoss}
    	Each $f_i:\mathbb{R}^d \rightarrow \mathbb{R}$ can be written as \(
    	f_i(x) = \ell(a_i^Tx,b_i), 
    	\) 
    	such that ${\rm supp}(\nabla f_i(x)) = {\rm supp}(a_i)$ for given data $\left\{(a_i,b_i)\right\}_{i=1}^m$ with $a_i\in\mathbb{R}^d$ and $b_i\in\mathbb{R}$. 
    \end{assumption}
    
    Assumption \ref{assum:SparsityEmpiricalLoss}, which is satisfied for standard empirical risk minimization problems, implies that the sparsity pattern of component function gradients can be computed off-line directly from the data. We will consider two important sparsity measures: the average and maximum conflict graph degree of the data, defined as
    \begin{align*}
    \Delta_{\rm ave}  &=  \frac{1}{m}\sum_{i=1}^m \left\{\sum_{j=1, j\neq i}^m
    \mathbf{1}\{ \mbox{supp}(a_i) \cap \mbox{supp} (a_j) \neq \emptyset \}\right\} \\  
    \Delta_{\max} & =  \max_{i\in [1, m]}    \left\{ \sum_{j=1, j\neq i}^m
    \mathbf{1}\{ \mbox{supp}(a_i) \cap \mbox{supp} (a_j) \neq \emptyset \} \right\}.
    \end{align*}
    As shown next, these sparsity measures allow us to derive a tighter bound $\bar{L}$ for the Lipschitz constant  of the total loss:
    \begin{lemma}
    	\label{lemma:LipschitzWholeFcn}
    	Consider the optimization problem \eqref{eqn:Problem} under Assumption~\ref{assum:SparsityEmpiricalLoss}. If $\ell$ has $L$-Lipschitz continuous gradient, then the gradient of the total loss is $\bar L$-Lipschitz continuous with \[
    	\bar L =  L\sqrt{m(1+\Delta)}, 
    	\]
    	where $\Delta = {\min}(\Delta_{\rm ave},\Delta_{\max})$. 
    \end{lemma}
    \begin{proof}
    	See Appendix \ref{app:lemma:LipschitzWholeFcn}.
    \end{proof}
    
    These sparsity measures are used to tighten our convergence results, especially in  Section \ref{sec:DQGD} and \ref{sec:DIAG}.

 %==================================================================================
    %
    % Quantized gradient method
    %
    %==================================================================================
    \section{Unbiased random quantization} \label{sec:UnbiasedQuantizer}
    
    In this paper, we are interested in optimization using unbiased randomized quantizers (URQs):  
    
    \begin{definition}
    \label{def:UnbiasedRandQuant}
    A mapping $Q:\mathbb{R}^d \rightarrow\mathbb{R}^d$ is called an \emph{unbiased random  quantizer} if, for every $v \in \mathbb{R}^d$,
    %\begin{align*}
    %   {\rm sign}(Q(v_k)) = {\rm sign}(v_k) 
    %\end{align*}
    %for $v_k\in\mathbb{R}^d$, and it satisfies
    %
    \begin{enumerate}
	    \item $\mathrm{supp}(Q(v)) \subseteq \mathrm{supp}(v)$
	    \item $\mathbf{E}\{Q(v)\} = v$
        \item $\mathbf{E}\{ \| Q(v) \|^2 \} \leq \alpha \| v \|^2 $ 
    \end{enumerate}
    for some finite positive $\alpha$. In addition, $Q$ is said to be \emph{sign-preserving} if 
    \begin{align*}
       [Q(v)]^i v^i \geq 0 
    \end{align*}
    for every $v\in \mathbb{R}^d$ and $i\in [1,d]$.
    \end{definition}
    
    Unbiased random quantizers  satisfy some additional useful inequalities. First, 
    \begin{align*}
    \mathbf{E}\, \{ \Vert Q(v)\Vert_0\} &\leq c, 
    %\end{align*}
    \intertext{for any $ v\in \mathbb{R}^d$ and a finite positive constant {\color{black}$c\leq d$}. The sign-preserving property guarantees the same direction between the compressed vector and the full one.  This property of $Q$ also implies that}
    %\begin{align*}
    	{ \mathbf{E}\, } \Vert Q(v) - v  \Vert{^2} &\leq \beta  \Vert v\Vert^2, 
    \end{align*}
    for any $ v\in \mathbb{R}^d$ and a finite positive constant $\beta  \leq \alpha-1$. As we will show next, it is typically possible to derive better bounds for $c$ and $\beta$ when we consider specific classes of gradient compressors.
     
     \subsection{Examples of unbiased random quantizers}
    Several randomized gradient compression algorithms have been proposed for distributed optimization problems under limited communications. Important examples include the \emph{gradient sparsifier}~\cite{wangni2017gradient}, the  \emph{low-precision quantizer}~\cite{alistarh2016qsgd}  and the \emph{ternary quantizer}~\cite{wen2017terngrad} defined below.
    % Both quantizers often produce sparse vectors. Since the gradient sparsifier tends to send the elements with large magnitude, it is the randomized and generalized version of the greedy coordinate descent method presented in \cite{nutini2015coordinate}. The low-precision quantizer sends elements with low precision depending on the quantization levels specified by the users.  These URQs have the following definitions.

    %-------------------------------------------------------------------------
    %
    % Definition of Gradient Sparsifier 
    %
    %-------------------------------------------------------------------------
    \begin{definition}
    \label{def:GradientSparsifier}
    The gradient sparsifier $S:\mathbb{R}^d \rightarrow\mathbb{R}^d$ is defined as 
    \begin{align*}
        {S^i(v)} = \left\{ \begin{array}{rl}  v^i/{ p^i} & \mbox{with probability } { p^i}  \\ 0 & \mbox{otherwise}\end{array} \right. , 
    \end{align*}
    where ${ p^i}$ is probability that coordinate $i$ is selected.
    \end{definition}
 
     Note that when the gradient sparsifier uses the same probability for each coordinate, it will effectively result in a randomized coordinate descent. Choosing $p^i=\vert v^i\vert/\Vert v\Vert$, on the other hand, will result in the ternary quantizer~\cite{wen2017terngrad}:
%    Note that the gradient sparsifier with $p_i^k =|v^i_k|/\mathop{\max}\limits_{i\in[1,d]} |v^i_k|$ is the generalized version of the %compression function proposed in \cite{wen2017terngrad}. 
    
        \begin{definition}
    	\label{def:TernaryQuant}
    	The ternary quantizer $T:\mathbb{R}^d \rightarrow\mathbb{R}^d$ is defined as 
    	\begin{align*}
    	{T^i(v)} = \left\{ \begin{array}{rl}  \| v \| {\rm sign}(v^i) & \mbox{with probability } {|v^i|}/{\|v\|}  \\ 0 & \mbox{otherwise}\end{array} \right. . 
    	\end{align*}
    	\end{definition}
    	
    %-------------------------------------------------------------------------
    %
    % Definition of low-precision quantizer
    %
    %-------------------------------------------------------------------------
     The low-precision quantizer~\cite{alistarh2016qsgd}, defined next, combines sparsification of the gradient vector with quantization of its element to further reduce the amount of information exchanged. 
    \begin{definition}
    \label{def:DanQuant}
    The low-precision quantizer $Q_b:\mathbb{R}^d \rightarrow\mathbb{R}^d$ is defined as 
    \begin{align*}
        {Q_b^i(v)} &= \|v\| {\rm sign}(v^i)\xi(v,i,s),  
    \intertext{where}
         \xi(v,i,s) &=  \left\{ \begin{array}{rl}  {l}/{s} & \mbox{with probability } 1-p\left({|v^i|}/{\|v \|},s\right) \\ {(l+1)}/{s} & \mbox{otherwise}   \end{array}   \right.,  
    \end{align*}
    and $p(a,s)  = as - l \mbox{ for any } a\in [0,1]$. Here,  $s$ is the number of quantization levels distributed between $0$ and $1$, and $l\in [0,s)$ such that ${|v^i|}/{\|v \|} \in [l/s,(l+1)/s]$. 
    \end{definition}

    Notice that when we let %$p_i^k = |v_i^k|/\|v_k\|$ for all $i,k$ in Definition \ref{def:GradientSparsifier} and when we let  
    $s=1$ (and hence $l=0$) in Definition \ref{def:DanQuant}, the low-precision quantizer also reduces to  the ternary quantizer defined above.    It is easily shown that these quantizers are sign-preserving unbiased random quantizers. Specifically, we have the following results:
    
 %-------------------------------------------------------------------------
    %
    % Properties of Gradient Sparsifier 
    %
    %-------------------------------------------------------------------------
     \begin{proposition}[\cite{wangni2017gradient}]
    \label{proposition:GradientSparsifier}
    The gradient sparisifier $S:\mathbb{R}^d\rightarrow \mathbb{R}^d$ is a sign-preserving URQ, which satisfies 
    
    \begin{enumerate}
        \item $\mathbf{E}\{ \| S(v) \|^2 \} \leq (1/p_{\min}) \| v \|^2$ where $p_{\min}= \mathop{\min}\limits_{i\in[1,d]}p^i$ , and
        \item $\mathbf{E}\{\| S(v)) \|_0\} =\sum_{i=1}^d p^i.$ 
       % \item $ \Vert S(v) - v \Vert \leq {\max}\left( |1-1/p_{\min}| ,1 \right) \|v\|$ with $p_{\min}=\mathop{\min}\limits_{i\in[1,d]}p^i$.
    \end{enumerate}
    
    \end{proposition}

      %-------------------------------------------------------------------------
    %
    % Properties of low-precision quantizer
    %
    %-------------------------------------------------------------------------
    \begin{proposition}[Lemma 3.4 in \cite{alistarh2016qsgd}]
    \label{proposition:DanQuant}
    The low-precision quantizer $Q_b:\mathbb{R}^d \rightarrow\mathbb{R}^d$ is a sign-preserving URQ, which satisfies 
    \begin{enumerate}
        \item $\mathbf{E}\{ \| Q_b(v) \|^2 \} \leq \left(1+\min\left({d}/{s^2},{\sqrt{d}}/{s} \right)\right) \| v \|^2,$   and
        \item $\mathbf{E}\{\| Q_b(v)) \|_0\} \leq s(s+\sqrt{d})$.
        %\item $\|Q_b(v) - v \|\leq \sqrt{d}/(2s)\| v\|$.
    \end{enumerate}
    \end{proposition}
 
   Proposition \ref{proposition:GradientSparsifier} and \ref{proposition:DanQuant} both imply that  $\mathbf{E}\| Q(v) \|^2 $ is close to $ \| v \|^2$ if the URQs are sufficiently accurate; \emph{e.g.}, when we set $p^{i} = 1$ for all $i$ in the gradient sparsifier (we send the full vector) and when  we let $s\rightarrow\infty$ in the low-precision quantizer (we send the exact solution). Although the probability $p^i$ in the gradient sparsifier can be time-varying (\emph{e.g.}, when we set $p^i \propto v^i$) we  assume a time-invariant $\alpha$-value in the analysis below to simplify notation.

  %%%%%%%%%%%%%%%%%%%%%%%%%%%%%%%%%%%%%%%%%%%%%%%%%%%%%%%%%%%%%%%%%%%%%%%%%%%%%%%%%%%%%%%%%%%%%%%
    % Convergence analysis of quantized gradient method 
    %%%%%%%%%%%%%%%%%%%%%%%%%%%%%%%%%%%%%%%%%%%%%%%%%%%%%%%%%%%%%%%%%%%%%%%%%%%%%%%%%%%%%%%%%%%%%%%
    \section{Convergence Analysis of Quantized Gradient Method}\label{sec:ConvQuanNonDist}
    
    In this section, we study the impact of gradient compression on the convergence rate guarantees for
the gradient descent algorithm. Although this single-master/single-worker architecture is of limited
practical interest, it complements and improves on earlier results (\emph{e.g.} [14]) and establishes a baseline
for the distributed master-worker architectures studied later. Explicit formulas for the iteration and
communication complexity of GD with URQ compression are also given.

    We start by considering the compressed GD algorithm
    \begin{equation}
        x_{k+1} = x_k - \gamma_k Q(\nabla f(x_k)),
        \label{eqn:QuantizedGD}
    \end{equation}    
    where $\gamma_k$ is a positive step size, and $Q:\mathbb{R}^d\rightarrow\mathbb{R}^d$ is  a URQ.  Throughout this section, we derive explicit expressions for how the variance bound $\alpha$ of the URQ affects admissible step-sizes and guaranteed convergence times. We  begin by considering strongly convex optimization problems.

    \begin{theorem}\label{thm:SCSerialQGD}
    Consider the optimization problem \eqref{eqn:Problem} under Assumption \ref{assum:FiLipschitz}, \ref{assum:mustronglyconvex} and \ref{assum:SparsityEmpiricalLoss}. Let $\bar L =L\sqrt{m(1+\Delta)}$ and $\Delta = {\min}(\Delta_{\rm ave},\Delta_{\max})$. Then, the iterates $\{x_k\}_{k\in \mathbb{N}}$ generated by \eqref{eqn:QuantizedGD} with $\gamma_k = (1/\alpha)\left(2/(\mu+\bar L)\right)$  satisfy 
    \begin{align*}
        \mathbf{E}\|x_{k} - x^\star \|^2 \leq  \rho^k \| x_0 - x^\star \|^2,
    \end{align*}
    where $\rho = 1-\frac{1}{\alpha}\frac{4\mu\cdot  \bar L}{(\mu+\bar L)^2}$.
    \end{theorem}

    \begin{proof}
    See Appendix \ref{app:thm:SCSerialQGD}.
    \end{proof}

    One naive encoding of a vector processed by the URQ requires $c(\log_2 d + B)$ bits: $\log_2 d$ bits to represent each index and $B$ bits to represent the corresponding vector entry of $c$ non-zero values. Hence, Theorem \ref{thm:SCSerialQGD} yields the following iteration and communication complexity.  
    \begin{corollary}\label{corr:SCSerialQGD}
    Consider the optimization problem \eqref{eqn:Problem} under Assumption \ref{assum:FiLipschitz}, \ref{assum:SparsityEmpiricalLoss} and \ref{assum:mustronglyconvex}. Let $\bar L =L\sqrt{m(1+\Delta)}$ and $\Delta = {\min}(\Delta_{\rm ave},\Delta_{\max})$.  Given $\varepsilon_0 = \| x_0 - x^\star\|^2$, by running \eqref{eqn:QuantizedGD} with $\gamma_k = (1/\alpha)\left(2/(\mu+\bar L)\right)$  for at most 
    \begin{align*}
        k^\star &=  \alpha\frac{(\mu+\bar L)^2}{4\mu\bar L}{\rm log}\left({\varepsilon_0}/{\varepsilon}\right) \\ \intertext{iterations, under which}
        B^\star &= \left( {\rm log}_2d + B \right) c \cdot \alpha\frac{(\mu+\bar L)^2}{4\mu\bar L}{\rm log}\left({\varepsilon_0}/{\varepsilon}\right) 
    \end{align*}
    bits are sent, we ensure that $ \mathbf{E}\|x_{k} - x^\star \|^2 \leq \varepsilon$. Here $B$ is the number of bits required to encode a single vector entry and $\mathbb{E}\{ \Vert Q(v)\Vert_0\} \leq c$. 
    \end{corollary}
    \begin{proof}
    See Appendix \ref{app:corr:SCSerialQGD}.
    \end{proof}

    Theorem \ref{thm:SCSerialQGD} quantifies how the convergence guarantees depend on $\alpha$.  If the worker node sends the exact gradient, \emph{i.e.} $Q(\nabla f(x_k)) = \nabla f(x_k)$,  $\alpha = 1$  and Theorem \ref{thm:SCSerialQGD} recovers the convergence rate result of GD for strongly convex optimization with $\gamma_k = 2/(\mu+\bar L)$ presented in \cite{polyak1987introduction,nesterov2013introductory}. If the quantizer produces a less accurate vector (larger $\alpha$), then we must decrease the step size $\gamma_k$ to guarantee numerical stability, and accept that the $\varepsilon$-convergence times $T^{\star}$ will increase.  The results above can also be extended to convex optimization problems:
    %including associated iteration and communication complexity are extended to convex optimization problems. 

%%%%%%%%%%%%%%%%%%%%%%%%%%%%%%%%%%%%%%%%%%%%%%%%%%%%%%%%%%%%%
%   Q GD convex optimization 
%%%%%%%%%%%%%%%%%%%%%%%%%%%%%%%%%%%%%%%%%%%%%%%%%%%%%%%%%%%%%
 \begin{theorem}\label{thm:QGDConvex}
    Consider the optimization problem \eqref{eqn:Problem} under Assumption \ref{assum:FiLipschitz} and \ref{assum:SparsityEmpiricalLoss}. Let $\Delta = {\min}(\Delta_{\rm ave},\Delta_{\max})$ and $\bar L =L\sqrt{m(1+\Delta)}$. Then, the iterates $\{x_k\}_{k\in \mathbb{N}}$ generated by \eqref{eqn:QuantizedGD} with $\gamma_k = (1/\bar L\alpha)$ satisfy 
    \begin{align*}
        \mathbf{E}\left( f(x_T) - f^\star \right) \leq \frac{\alpha \bar L}{2(T+1)}\| x_0 - x^\star \|^2.
    \end{align*}
    \end{theorem}
\begin{proof}
See Appendix \ref{app:thm:QGDConvex}.
\end{proof}

  \begin{corollary} \label{corr:QGDConvex}
     Consider the optimization problem \eqref{eqn:Problem} under Assumption \ref{assum:FiLipschitz} and \ref{assum:SparsityEmpiricalLoss}.  Let  $\bar L =L\sqrt{m(1+\Delta)}$ and $\Delta = {\min}(\Delta_{\rm ave},\Delta_{\max})$. Given $\varepsilon_0 = \| x_0 - x^\star\|^2$, by running \eqref{eqn:QuantizedGD} with $\gamma_k = (1/\bar L\alpha)$ for at most 
     \begin{align*}
         T^\star &= \frac{\alpha\bar L}{2}\cdot \frac{\varepsilon_0}{\varepsilon} \\ \intertext{iterations, under which}
         B^\star &= \left( {\rm log}_2d + B \right)c \cdot \frac{\alpha\bar L}{2}\cdot \frac{\varepsilon_0}{\varepsilon} 
     \end{align*}
     bits are sent, we ensure  $ \mathbf{E}\left( f(x_T) - f^\star \right) \leq \varepsilon$.  Here $B$ is the number of bits required to encode a single vector entry and $\mathbb{E}\{ \Vert Q(v)\Vert_0\} \leq c$. 
    \end{corollary}
    \begin{proof}
See Appendix \ref{app:corr:QGDConvex}.
\end{proof}
    
    We conclude this section by studying the following 
    compressed IAG algorithm: given an initial point $x_0$ and a fixed, positive step size $\gamma$
    \begin{equation}
    x_{k+1} = x_k - \gamma Q\left( \sum_{i=1}^m \nabla f_i( x_{k-\tau_k^i} ) \right).  
    \label{eqn:QIAGnonDistr}
    \end{equation}
    The iteration accounts for heterogeneous worker delays, but performs a centralized compression of the sum of staled gradients. We include the result here to highlight how the introduction of heterogeneous delays affect our convergence guarantees, and consider it as an intermediate step towards the more practical architectures studied in the next section. 
    Note that if we let $\tau_k^i = 0$ (and therefore $\tau=0$), then the compressed IAG iteration \eqref{eqn:QIAGnonDistr} reduces to the compressed GD iteration \eqref{eqn:QuantizedGD}. 

     \begin{theorem}
    \label{thm:QthenIAGSCTrickHamid}
    Consider the optimization problem \eqref{eqn:Problem} under Assumption \ref{assum:FiLipschitz}, \ref{assum:SparsityEmpiricalLoss} and \ref{assum:mustronglyconvex}. Let  $\bar L =L\sqrt{m(1+\Delta)}$ and $\Delta = {\min}(\Delta_{\rm ave},\Delta_{\max})$, and suppose that $0< \gamma < \bar\gamma$ where
    \begin{align*}
        \bar  \gamma = \min \left(\frac{\mu}{\sqrt{\alpha}\tau \bar{L}^2},\frac{1}{\alpha \bar L}\right) 
     \end{align*}
     and $\tau_k^i \leq \tau$ for all $i,k$. Then, the iterates $\{x_k\}_{k\in\mathbb{N}}$ generated by \eqref{eqn:QIAGnonDistr} satisfy 
     \begin{align*}
      \mathbf{E} [f(x_k) -f(x^\star)] & \leq \left( p+q  \right)^{k/(1+2\tau)} \bigl(f(x_0) - f(x^\star)\bigr)  
     \end{align*}
     where $p = 1- \mu\gamma$ and $q= \bar L^4 \gamma^3 \tau^2 \alpha/\mu$. 
    \end{theorem}
    \begin{proof}
    See Appendix \ref{app:thm:QthenIAGSCTrickHamid}.
    \end{proof}

 The admissible step-sizes in Theorem~\ref{thm:QthenIAGSCTrickHamid} depends on both the delay bound $\tau$ and $\alpha$.  The upper bound on the step-size in Theorem \ref{thm:QthenIAGSCTrickHamid} is smaller than the corresponding result in Theorem~\ref{thm:SCSerialQGD}.  If the quantizer produces the exact output, then the proposed algorithm coincides with the IAG algorithm~\eqref{eqn:IAG} for strongly convex optimization. Suppose that $\alpha =1, \mu/\bar L \leq \tau$, and $\gamma = 0.5\bar\gamma$. Then, the IAG iteration satisfies
\begin{align*}
f(x_k) - f(x^\star) \leq \left(1-\frac{1}{8}\frac{1}{1+2\tau}\frac{\mu^2}{\tau \bar L^2} \right)^{k}\bigl( f(x_0) - f(x^\star) \bigr)
\end{align*}
where the inequality follows from the fact that $(1-x)^a \leq 1-ax$ for $x,a\in[0,1].$ Thus, our step-size is more than three times larger than the one derived in~\cite{gurbuzbalaban2017convergence}, which results in corresponding improvements in convergence factors. 

Next, Theorem~\ref{thm:QthenIAGSCTrickHamid} estimates the associated $\varepsilon$-convergence times and expected information exchange from workers to master.
  \begin{corollary}\label{corr:QthenIAGSCTrickHamid}
    Consider the optimization problem \eqref{eqn:Problem} under Assumption \ref{assum:FiLipschitz}, \ref{assum:SparsityEmpiricalLoss} and \ref{assum:mustronglyconvex}.  Let  $\bar L =L\sqrt{m(1+\Delta)}$ and $\Delta = {\min}(\Delta_{\rm ave},\Delta_{\max})$, and suppose that
    \begin{align*}
         \gamma < \min \left(\frac{\mu}{\sqrt{\alpha}\tau \bar{L}^2},\frac{1}{\alpha \bar L}\right),
     \end{align*}
     where $\tau_k^i \leq \tau$ for all $i,k$.  Given $\varepsilon_0 = f(x_0)-f^\star$, by running \eqref{eqn:QIAGnonDistr} for at most
    \begin{align*}
        k^\star & = (1+2\tau)\frac{\mu}{\gamma \left( \mu^2 - \bar L^4 \gamma^2 \tau^2 \alpha  \right)} {\log}(\varepsilon_0/\varepsilon)\\ \intertext{iterations, under which}
        B^\star & = ({\log}_2 d +B )c \cdot  (1+2\tau)\frac{\mu}{\gamma \left( \mu^2 - \bar L^4 \gamma^2 \tau^2 \alpha  \right)} {\log}(\varepsilon_0/\varepsilon)
    \end{align*}
    bits are sent,  we ensure $\mathbf{E}\left( f(x_k) - f^\star \right)\leq\varepsilon$.  Here $B$ is the number of bits required to encode a single vector entry and $\mathbb{E}\{ \Vert Q(v)\Vert_0\} \leq c$. 
    \end{corollary}
\begin{proof}
See Appendix \ref{app:corr:QthenIAGSCTrickHamid}.
\end{proof}

    % if $\beta,\alpha_{\max}$ is sufficiently small such that $2(1+\beta)\tau + \alpha_{\max} \leq \tau/a$ where $a=8/25$. In addition, %Corollary \ref{corr:QthenIAGSCTrickHamid} with $\tau=0$ implies higher number of iterations to reach a certain solution accuracy %than Corollary \ref{corr:SCSerialQGD} by $mL/\mu$.  
    
    Furthermore, we extend the result for the optimization problem without the strong convexity assumption as follows:  
    
    %%---------------------------------------------------------------------
    %%---------------------------------------------------------------------
    %%---------------------------------------------------------------------
    % non-convex optimization problem: Compressed IAG
    %%---------------------------------------------------------------------
    %%---------------------------------------------------------------------
    %%---------------------------------------------------------------------
    \begin{theorem}\label{thm:CompressedIAGNonstronglyConvex}
     Consider the optimization problem \eqref{eqn:Problem} under Assumption \ref{assum:FiLipschitz} and  \ref{assum:SparsityEmpiricalLoss}. Let  $\bar L =L\sqrt{m(1+\Delta)}$ and $\Delta = {\min}(\Delta_{\rm ave},\Delta_{\max})$, and suppose that
\begin{align*}
    \gamma < \frac{1}{\sqrt{1+8\left( 1 + \beta(1+\theta) \right)\tau(\tau+1)}}\frac{2}{\bar L}, \end{align*} and \(\beta < 1/\left(2(1+1/\theta)\right)\) for $\theta>0$. Then, the iterates $\{x_k\}_{k\in\mathbb{N}}$ generated by \eqref{eqn:QIAGnonDistr} satisfy 
    \begin{align*}
      \mathop{\min}\limits_{k\in[0,K]} \mathbf{E}\| \nabla f(x_k) \|^2 \leq \frac{1}{a}\frac{1}{K+1}\left(f(x_0) - f^\star\right),
    \end{align*}
    where $a = \gamma/2-\gamma\beta(1+1/\theta). $
    \end{theorem}
    \begin{proof}
    See Appendix \ref{app:thm:CompressedIAGNonstronglyConvex}.
    \end{proof}
    
    Theorem \ref{thm:CompressedIAGNonstronglyConvex} implies the sufficient accuracy of the compression techniques to guarantee the numerical stability of the compressed IAG algorithms. Unlike Theorem \ref{thm:QthenIAGSCTrickHamid}, the step size from this theorem is independent of the conditional number $\bar L/\mu$.

    \section{Distributed Quantized Gradient Method}\label{sec:DQGD}
    \iffalse
    \begin{algorithm}
        \KwData{step size $\gamma_k$, the initial point $x_0$, the quantizer $Q$}
        
       % \CommentSty{// $x_t$ accumulates all the gradients generated up to the end of iteration $t$}\;

       %\CommentSty{// $v_t$ accumulates all the gradients \emph{applied} up to the end of iteration $t$}\;

        % $x_0 = v_0 = \vec{0}$\;
        
        \For{each step $k$}
        {
            %\CommentSty{// For simplicity, assume the gradient sign is inverted}
            Each worker computes $\nabla f_i (x_k)$ and sends $Q(\nabla f_i(x_k))$ to the master\;
            The master computes $\Delta_{\max}^k$ and $\gamma_k$ \; 
            The master computes $x_{k + 1} \gets x_k - \gamma_k \sum_{i=1}^m Q(\nabla f_i(x_k))$\;
        }
    \caption{Synchronous distributed quantized gradient algorithm}
    \label{algo:syndisQGD}
    \end{algorithm}
    \fi
   
   Before we present the convergence results for the compressed  incremental aggregate  gradient algorithm, we consider its synchronous counterpart where the master waits for all workers to return before it updates the decision vector. Thus, we study the following algorithm: given the initial point $x_0$, a positive step size $\gamma_k$ and the URQ $Q$, iterates $x_k$ are generated via
    \begin{equation}
        x_{k+1} = x_k - \gamma_k \sum_{i=1}^m Q(\nabla f_i(x_k)).
        \label{eqn:syndisQGD}
    \end{equation}
	Since URQs are random and modify the gradient vectors and their support, the sparsity patterns of the quantized gradients are time-varying and can be characterized by the quantities 
	%and by Assumption~\ref{assum:SparsityEmpiricalLoss} 
	% be computed off-line. Instead, a careful analysis should consider the time-varying quantities 
	\begin{equation}\label{eqn:QuantizedSparsityMeas}
    \begin{array}{rl}
      \Delta _{\max }^k & = {\mathop{\max}\limits_{i \in [1,m]}}\left\{ {\sum\limits_{j = 1,j \ne i}^m {\mathbf{1}} \left\{ {{\text{supp}}(Q({a_i})) \cap {\text{supp}}(Q({a_j})) \ne \emptyset } \right\}} \right\} \\ 
      \Delta_{\rm ave}^k &= \frac{1}{m}\sum\limits_{i = 1}^m {\left\{ {\sum\limits_{j = 1,j \ne i}^m {\mathbf{1}} \{ {\text{supp}}(Q({a_i})) \cap {\text{supp}}(Q({a_j})) \ne \emptyset \} } \right\}} . 
    \end{array}
    \end{equation}
	A limitation with these quantities is that they cannot be computed off-line. 
    However, 
    since gradient compression reduces the support of vectors, ${\text{supp}}(Q({a_i}))\subset {\rm supp}(a_i)$, it always holds that  $\Delta_{\max}^k\leq \Delta_{\max}$ and $\Delta_{\rm ave}^k \leq \Delta_{\rm ave}$.

    The next lemma enables us to benefit from sparsity in our analysis. 
    \begin{lemma}
    \label{lemma:datasparsityTrick}
    Under Assumption \ref{assum:SparsityEmpiricalLoss}, for $k\geq 0$ 
    \begin{align*}
    \left\|  \sum_{i=1}^m Q(\nabla f_i(x_k)) \right\|^2  & \leq \sigma_k\sum_{i=1}^m \left\| Q(\nabla f_i(x_k)) \right\|^2,
    \end{align*}
    where
    \begin{align*}
    \sigma_k &= {\min}\left( \sqrt{m(1+\Delta_{\rm ave}^k)},1+\Delta_{\max}^k \right).
    \end{align*}
    Moreover, 
    \begin{align*}
    \sigma_k \leq \sigma = {\min}\left( \sqrt{m(1+\Delta_{\rm ave})},1+\Delta_{\max} \right).
    \end{align*}  
    \end{lemma}
    \begin{proof}
    See Appendix \ref{app:lemma:datasparsityTrick}.   
    \end{proof}
    Notice that Lemma  \ref{lemma:datasparsityTrick} quantifies the combined impact of data sparsity and compression. We have $\sigma_k =1$ if the quantized gradients are completely sparse (their support sets do not overlap), whereas $\sigma_k = m$ if the quantized gradients are completely dense (all support sets overlap). 
    
    We are now ready to state our convergence result for strongly convex loss functions.
\iffalse
       \begin{theorem}
    \label{thm:DistributedQGDSC}
    Consider the optimization problem \eqref{eqn:Problem}under Assumption \ref{assum:FiLipschitz}, \ref{assum:SparsityEmpiricalLoss} and \ref{assum:mustronglyconvex}.  Suppose that $\gamma_k = 1/\left(L\alpha (1+\theta) \sigma_k \right)$ where $\theta>0$. Then, the iterates $\{x_k\}_{k\in \mathbb{N}}$ generated by \eqref{eqn:QuantizedGD} satisfy 
    \begin{align*}
        \mathbf{E}\|x_k-x^\star \|^2 \leq {\rho}^k_{\max} \| x_0 - x^\star\|^2 + \frac{1}{1-\rho_{\max}}\frac{\gamma_{\max}}{\theta L}  \sum_{i=1}^m \|\nabla f_i(x^\star) \|^2,
    \end{align*}
    %
    where $\rho_k \leq \rho_{\max}$, $\rho_{\max} \in (0,1)$, $\rho_k = 1- \mu\gamma_k$, and $\gamma_k \leq \gamma_{\max}$ for all $k$. 
    \end{theorem}
    \begin{proof}
    See Appendix \ref{app:thm:DistributedQGDSC}.
    \end{proof}
\fi

       \begin{theorem}
	\label{thm:DistributedQGDSC}
	Consider the optimization problem \eqref{eqn:Problem} under Assumption \ref{assum:FiLipschitz},  \ref{assum:mustronglyconvex} and \ref{assum:SparsityEmpiricalLoss}.  Suppose that $\gamma = 1/\left(L\alpha (1+\theta) \sigma \right)$ for some $\theta>0$. Then, the iterates $\{x_k\}_{k\in \mathbb{N}}$ generated by \eqref{eqn:QuantizedGD} satisfy 
	\begin{align*}
	\mathbf{E}\|x_k-x^\star \|^2 \leq (1-\mu\gamma)^k \| x_0 - x^\star\|^2 +\frac{1}{\mu \theta L}  \sum_{i=1}^m \|\nabla f_i(x^\star) \|^2,
	\end{align*}
\end{theorem}
\begin{proof}
	See Appendix \ref{app:thm:DistributedQGDSC}.
\end{proof}

    \iffalse
     \begin{theorem}
    \label{thm:DistributedQGDConvex}
    Consider the optimization problem \eqref{eqn:Problem} under Assumption \ref{assum:FiLipschitz} and \ref{assum:ConvexFcn}. Suppose that $\gamma_k = 1/\left(L\alpha (1+\theta) (1+\Delta_{\max}^k)\right)$ where $\theta>0$. Then, the iterates $\{x_k\}_{k\in \mathbb{N}}$ generated by \eqref{eqn:QuantizedGD} satisfy 
    \begin{align*}
       \mathbf{E}(f( \bar{x}_T) - f(x^\star)) \leq \frac{1}{\gamma_{\min}}\frac{1}{T}\|x_0 - x^\star \|^2 + \frac{1}{\theta L_{\max}}\sum_{i=1}^m\mathbf{E} \|\nabla f_i (x^\star)\|^2,
    \end{align*}
    %
    where $\gamma_k \geq \gamma_{\min}$ for all $k$ and $\bar x_T = \frac{1}{T}\sum_{k=0}^{T-1}x_k.$
    \end{theorem}
    \begin{proof}
    See Appendix \ref{app:thm:DistributedQGDConvex}. 
    \end{proof}
    \fi 
    Theorem \ref{thm:DistributedQGDSC}  states that the iterates genenrated by D-QGD \eqref{eqn:syndisQGD} converge to a ball around the optimal solution. It shows explicitly how the sparsity measure $\sigma$ and the quantizer accuracy parameter $\alpha$ affect the convergence guarantees. Note that a larger value of $\theta$ allows for larger step-sizes and better convergence factor, but also a larger residual error.

   For simplicity of notation and applicability of the results, we formulated Theorem~\ref{thm:DistributedQGDSC} in terms of $\sigma$ and not $\sigma_k$  (the proof, however,  also provides convergence guarantees in terms of $\sigma_k$). The result is conservative in the sense that compression increases sparsity of the gradients, which should translate into larger step-sizes. To evaluate the degree of conservatism, we carry out Monte Carlo simulations on the data sets described in Table~\ref{tab:data_experiment}. We indeed note that $\sigma_k$ is significantly smaller than $\sigma$. Next, we extend the result to convex optimizization problems. 
 
\begin{theorem}
    \label{thm:DistributedQGDConvex}
    Consider the optimization problem \eqref{eqn:Problem} under Assumption \ref{assum:FiLipschitz} and \ref{assum:SparsityEmpiricalLoss}. Let  $\sigma = \min\left( \sqrt{m(1+\Delta_{\rm ave})} , 1+ \Delta_{\max} \right)$ and $\theta > 0$. Suppose that $\gamma = 1/\left(L\alpha (1+\theta) \sigma \right)$. Then, the iterates $\{x_k\}_{k\in \mathbb{N}}$ generated by \eqref{eqn:QuantizedGD} satisfy 
    \begin{align*}
       \mathbf{E}(f( \bar{x}_T) - f(x^\star)) \leq \frac{1}{\gamma_{\min}}\frac{1}{T}\|x_0 - x^\star \|^2 + \frac{1}{\theta L}\sum_{i=1}^m\|\nabla f_i (x^\star)\|^2,
    \end{align*}
    where $\bar x_T = \frac{1}{T}\sum_{k=0}^{T-1}x_k.$
    \end{theorem}
\begin{proof}
See Appendix  \ref{app:thm:DistributedQGDConvex}.
\end{proof}

    \begin{table}[t]
    \centering
        \begin{tabular}{|c|c|c|c|c|}
        \hline 
                         &   & \multicolumn{3}{|c|}{${\mathbf{E}\{\sigma_k\}}/{m}$} \\ \cline{3-5}
            Data Set &$\sigma/m$ &  GS &  TQ &  LP \\ 
            \hline \hline 
            RCV1-train & $0.83$ & $0.66$  & $0.07$ & $0.42$  \\ 
            real-sim & $0.8278$ & $0.58$ & $0.06$ & $0.37$\\ 
            GenDense & $1$ & $1$ & $0.7$&  $1$\\ \hline 
        \end{tabular}
        \caption{Empirical evaluations of $\sigma_k$ and $\sigma$ with gradient sparsifier (GS) with $p_i=0.5$, with ternary quantizer (TQ), and with low-precision quantizer (LP) with $s=4$. }
        \label{tab:sigma_k}
    \end{table}

    %==============================================================================
    %
    % Proof Sketch 
    %
    %==============================================================================
    \section{Q-IAG Method} \label{sec:DIAG}
    In this section, we rather consider the quantized version of the optimization algorithm which is suited for communications with limited bandwidth. Therefore, we study the convergence rate of the quantized version of the IAG algorithm (Q-IAG) where the update is 
    \begin{equation}
        x_{k+1} = x_k - \gamma  \sum_{i=1}^m Q\left( \nabla f_i (x_{k-\tau_k^i}) \right), 
        \label{eqn:QIAG}
    \end{equation}
    where $\gamma$ is the constant step size, and $Q$ is the URQ.  Notice that \[\mathbf{E} \left\{ \sum_{i=1}^m Q\left( \nabla f_i (x_{k-\tau_k^i}) \right) \right\} =  \sum_{i=1}^m\nabla f_i( x_{k-\tau_k^i} ).\]

    By Assumption \ref{assum:SparsityEmpiricalLoss}, ${\rm supp}( Q(\nabla f_i(x_{k-\tau^i_k}))) = {\rm supp}(Q(a_i))$, and thus the sparsity measures defined \eqref{eqn:QuantizedSparsityMeas} will be used to strengthen our main analysis.

    Now, we present the result for strongly convex optimization.

    \begin{theorem}
    \label{thm:QIAGDistSC}
    Consider the optimization problem \eqref{eqn:Problem} under Assumption \ref{assum:FiLipschitz}, \ref{assum:SparsityEmpiricalLoss} and \ref{assum:mustronglyconvex}. Let  $\bar L =L\sqrt{m(1+\Delta)}$, $\Delta = {\min}(\Delta_{\rm ave},\Delta_{\max})$, and suppose that $0<\gamma<\bar\gamma$  where
    \begin{align*}
       \bar \gamma = \frac{2\mu}{1 + m\sigma\alpha L^2\left( 2\bar L^2\tau^2 + (1+\theta)  \right)}
    \end{align*}
    and  $\sigma = \min\left(  \sqrt{m(1+\Delta_{\rm ave})}, 1+\Delta_{\max}\right) $ and $\tau_k^i \leq \tau$ for all $i,k$, and $\theta>0$. Then, the iterates $\{x_k\}_{k\in \mathbb{N}}$ generated by \eqref{eqn:QIAG} satisfy 
    \begin{align*}
      \mathbf{E}\|x_k - x^\star \|^2 & \leq (p+q)^{k/(1+2\tau)} \|x_0 - x^\star \|^2 + e/(1-p-q), \\ \intertext{where}
    p & = 1 - 2\mu\gamma + \gamma^2   \\ 
        q& = 2m\sigma\alpha  L^2 \gamma^2 \bar L^2\tau^2 +  (1+\theta)\gamma^2 m\alpha\sigma L^2\\
        e & = \left(2m\alpha \gamma^2 \bar L^2\tau^2 + (1+1/\theta)\gamma^2 \sigma  \alpha\right) \sum_{i=1}^m   \left\|     \nabla f_i (x^\star) \right\|^2.
        %\\ 
        % V_k & = 
    \end{align*}
    \end{theorem}
    \begin{proof}
    See Appendix \ref{app:thm:QIAGDistSC}.
    \end{proof}
    
    Unlike the result for the compressed IAG algorithm \eqref{eqn:QIAGnonDistr}, Theorem \eqref{thm:QIAGDistSC} can only guarantee that the Q-IAG algorithm \eqref{eqn:QIAG} converges to a ball around the optimum.  Letting $\theta= 1$ and $\gamma = 0.5\bar\gamma$ in Theorem~\ref{thm:QIAGDistSC} yields the convergence bound 
\begin{align*}
\mathbf{E}\|x_k - x^\star \|^2 & \leq \left( 1-\frac{\mu^2}{1 + 2m\sigma\alpha L^2\left( \bar L^2\tau^2 + 1  \right)}    \right)^{k/(1+2\tau)}\| x_0 - x^\star \|^2 + E, \intertext{where}
E &= 2\mu\frac{m\alpha \bar L^2\tau^2 + \sigma\alpha}{1+2m\sigma\alpha L^2 (\bar L^2 \tau^2 + 1)} \sum_{i=1}^m \| \nabla f_i(x^\star) \|^2.
\end{align*}
Thus, the convergence rate and step-size for~\eqref{eqn:QIAG} depend on the delay bound $\tau$ and the URQ parameter $\alpha$. In particular, the convergence factor is penalized roughly by $\mu^2/(\alpha \bar L^4 \tau^2)$ when individual workers compress their gradient information. In the absence of the worker asynchrony ($\tau = 0$), the upper bound on the step-size becomes $\mu/(m\sigma\alpha L^2)$, which is smaller than the step-size allowed by Theorem \ref{thm:DistributedQGDSC} with $\theta=1$.

Next, we present the result for optimization problems without the strong convexity assumption on the objective function $f$. However, in this case we need to assume that the component functions have uniformly bounded gradients:
%impose the bounded gradient assumption for analyze the Q-IAG algorithms under this problem setting.
    \begin{assumption}\label{assum:BoundedGradient}
    There exists a scalar $C$ such that 
    \begin{align*}
        \| \nabla f_i(x) \| \leq C,
    \end{align*} 
    for any component function $f_i:\mathbb{R}^d\rightarrow\mathbb{R}$ and $x\in\mathbb{R}^d$. 
    \end{assumption} 
    
 One popular problem which satisfies Assumption~\ref{assum:BoundedGradient} is the low-rank least-squares matrix completion problem which arises Euclidean distance estimation, clustering and other applications~\cite{de2015taming,candes2009exact}.  Now, the result is shown below: 
    \begin{theorem}\label{thm:DistIAGHamidNSC}
      Consider the optimization problem \eqref{eqn:Problem} under Assumption \ref{assum:FiLipschitz}, \ref{assum:SparsityEmpiricalLoss} and \ref{assum:BoundedGradient}. Let  $\bar L =L\sqrt{m(1+\Delta)}$ and $\Delta = {\min}(\Delta_{\rm ave},\Delta_{\max})$, and suppose that \begin{align*}
        \gamma < \frac{1}{1+\sqrt{1+8\tau(\tau+1)}}\frac{2}{\bar L}, 
    \end{align*}
    and $\tau_k^i \leq \tau$ for all $i,k$. Then, the iterates $\{x_k\}_{k\in\mathbb{N}}$ generated by \eqref{eqn:QIAG} satisfy 
    \begin{align*}
         \mathop{\min}\limits_{k\in[0,K]} \mathbf{E}\| \nabla f(x_k) \|^2 &\leq \frac{2}{\gamma}\frac{1}{K+1}\left( f(x_0) -f^\star \right) + e, 
    \end{align*}
    where $e = 2\beta\sigma m C^2$. 
    \end{theorem}
    
    \begin{proof}
    See Appendix \ref{app:thm:DistIAGHamidNSC}.

    \end{proof}
    
    Unlike Theorem \ref{thm:QIAGDistSC}, the step size stated in 
    Theorem \ref{thm:DistIAGHamidNSC} does not depend on the condition number $\bar L/\mu$.

  %%%%%%%%%%%%%%%%%%%%%%%%%%%%%%%%%%%%%%%%%%%%%%%%%%%%%%%%%%%%%%%%%%%%%%%%%%%%%%%%%%%%%%%%%%
    %
    % Simulation results
    %
    %%%%%%%%%%%%%%%%%%%%%%%%%%%%%%%%%%%%%%%%%%%%%%%%%%%%%%%%%%%%%%%%%%%%%%%%%%%%%%%%%%%%%%%%%%
    \section{Simulation Results}\label{sec:exp}

    {\color{black}
   % We consider the empirical risk minimization problem \eqref{eqn:Problem} with component loss functions \[f_i(x) = \frac{1}{2m}\|A_i x - b_i \|^2 +\frac{\sigma}{2}\|x\|^2,\] 
    
    %where $A_i \in \mathbb{R}^{p\times d}$ and $b_i \in\mathbb{R}^p$ and $p$ is the number of data points which are stored in worker $i$. 
    
    We consider the empirical risk minimization problem \eqref{eqn:Problem} with component loss functions on the form of 
    \begin{align*}
        f_i(x) = \frac{1}{2\rho}\| A_i x - b_i \|^2 + \frac{\sigma}{2}\|x\|^2,
    \end{align*}
    %\eqref{eqn:ObjectiveFcnUsed}
    where $A_i\in\mathbb{R}^{p\times d}$ and $b_i \in\mathbb{R}^p$. We distributed data samples $(a_1,b_1),\ldots,(a_n,b_n)$ among $m$ workers. Hence, $n = mp$. The experiments were done using both synthetic and real-world data sets  as shown in Table \ref{tab:data_experiment}.  Each data sample $a_i$ is then normalized by its own Euclidean norm. We evaluated the performance of the distributed gradient algorithms \eqref{eqn:QuantizedGD}-\eqref{eqn:QIAG} using the gradient sparsifier, the low-precision quantizer and the ternary quantizer in Julia.  We set $m=3$, $x_0 = \mathbf{0}$, set $\sigma = 1$, and set $\rho$ equal to the total number of data samples according to Table \ref{tab:data_experiment}. In addition, GenDense from Table \ref{tab:data_experiment} generated the dense data set such that each element of the data matrices $A_i$ is randomly drawn from a uniform random number between $0$ and $1$, and each element of the class label vectors $b_i$ is the sign of a zero-mean Gaussian random number with unit variance.   For the gradient sparsifier, we assumed that vector elements are represented by $64$ bits (IEEE doubles) while the low-precision quantizer only requires $1+\log_2(s)$ bits to encode each vector entry. For the distributed algorithms, we have used $\tau=m$.

    \begin{table}[ht]
        \centering
        \begin{tabular}{|c|c|c|c|}
        \hline 
        Data Set    & Type & Samples & Dimension \\
        \hline \hline 
        RCV1-train  & sparse   & $23149$  & $47236$ \\ 
        \hline 
        real-sim    & sparse &  $72309$   & $20958$ \\ 
        \hline 
        covtype     & dense & $581012$   & $54$ \\
        \hline 
        GenDense   & dense & $40000$    & $1000$ \\
        \hline 
        \end{tabular}
        \caption{Summary of synthetic and real-world data sets used in our experiments.}
        \label{tab:data_experiment}
    \end{table}

    Figure \ref{fig:CGIter} and \ref{fig:CGCoordinates}  show the trade-off between the convergence in terms of iteration count and the number of communicated bits. Naturally, the full gradient method has the fastest convergence, and the ternary quantizer is slowest. The situation is reversed if we judge the convergence relative to the number of communicated bits. In this case, the ternary quantizer makes the fastest progress per information bit, followed by the $3$-bit low-precision quantizer ($s=4$). In fact, the full gradient descent requires more bits in the order of magnitude to make $50\%$ progress than the ternary quantizer.
    
	The corresponding results for Q-IAG in the asynchronous parameter server setting are shown in Figure~\ref{fig:distCIAGIter} and \ref{fig:distCIAGCoordinates}. The results are qualitatively similar: sending the gradient vectors in higher precision yields the fastest convergence but can be extremely wasteful in terms of communication load. The low-precision quantizer allows us to make a gentle trade-off between the two objectives, having both a rapid and communication-efficient convergence. In particular, the results from covtype show that a fast convergence in terms of both iteration counts and communications load for the low-precision quantizer with the higher number of quantization levels.

	%The plot also shows the impact of the free parameter $\theta$: as can be expected from our theoretical results, a higher value of $\theta$ leads to smaller step-sizes and slower convergence. It is only in the high accuracy regime where it could potentially be beneficial to use $\theta\geq 1$, potentially combined with an adaptive choice of quantizer accuracy.

\begin{figure}
\begin{subfigure}{.5\linewidth}
\centering
\includegraphics[scale = 0.44]{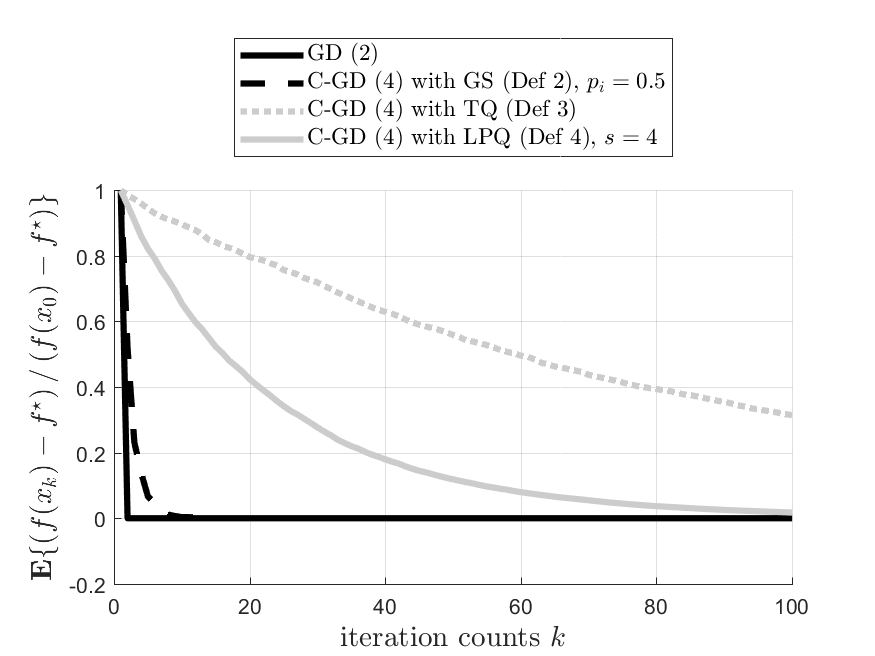}
\caption{}
\label{fig:CG_IterRealsim}
\end{subfigure}%
\begin{subfigure}{.5\linewidth}
\centering
\includegraphics[scale = 0.44]{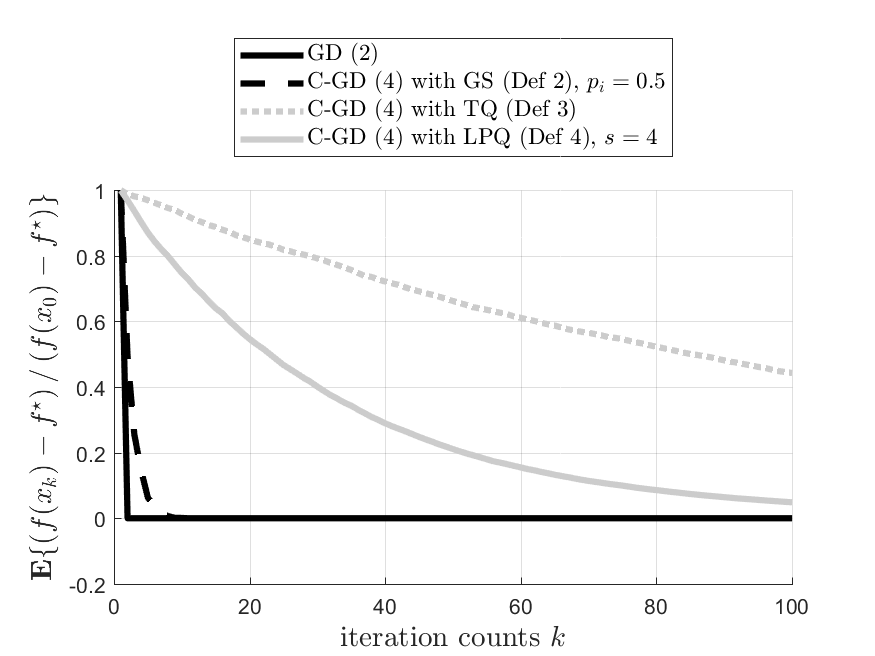}
\caption{}
\label{fig:CG_IterRCV1}
\end{subfigure}\\[1ex]
\begin{subfigure}{\linewidth}
\centering
\includegraphics[scale = 0.44]{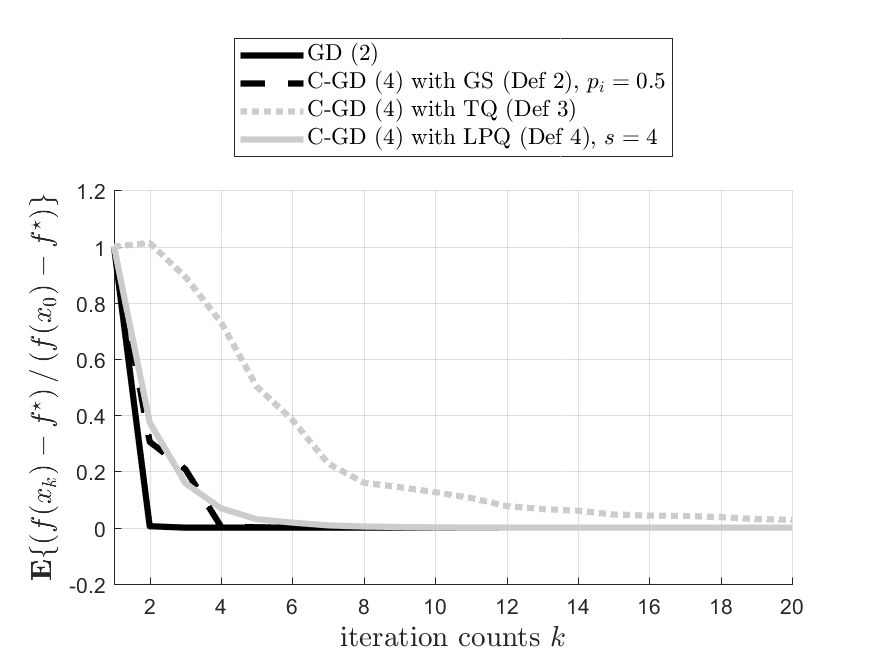}
\caption{}
\label{fig:CG_IterCovtype}
\end{subfigure}
\caption{Convergence of compressed gradient descent algorithms \eqref{eqn:QuantizedGD} using different compression techniques over real-world data sets; that is, (a) real-sim, (b) RCV1-train and (c) covtype. }
\label{fig:CGIter}
\end{figure}

\begin{figure}
\begin{subfigure}{.5\linewidth}
\centering
\includegraphics[scale = 0.44]{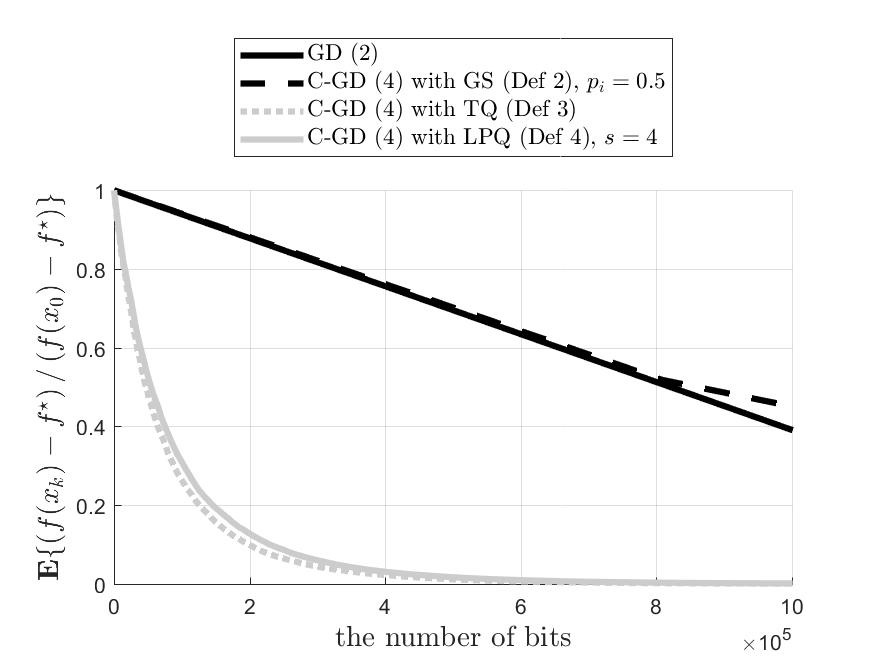}
\caption{}
\label{fig:CG_CoordRealsim}
\end{subfigure}%
\begin{subfigure}{.5\linewidth}
\centering
\includegraphics[scale = 0.44]{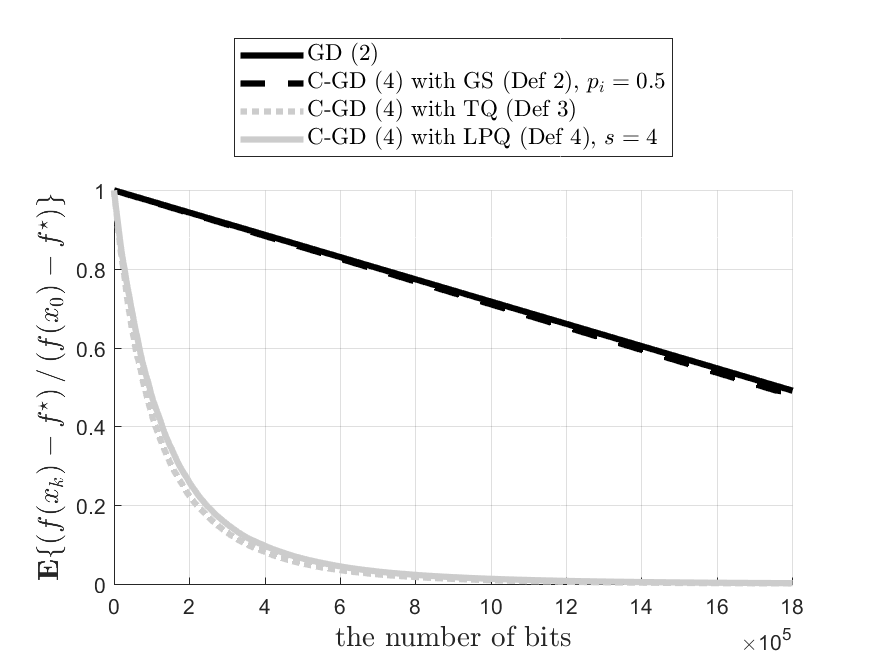}
\caption{}
\label{fig:CG_CoordRCV1}
\end{subfigure}\\[1ex]
\begin{subfigure}{\linewidth}
\centering
\includegraphics[scale = 0.44]{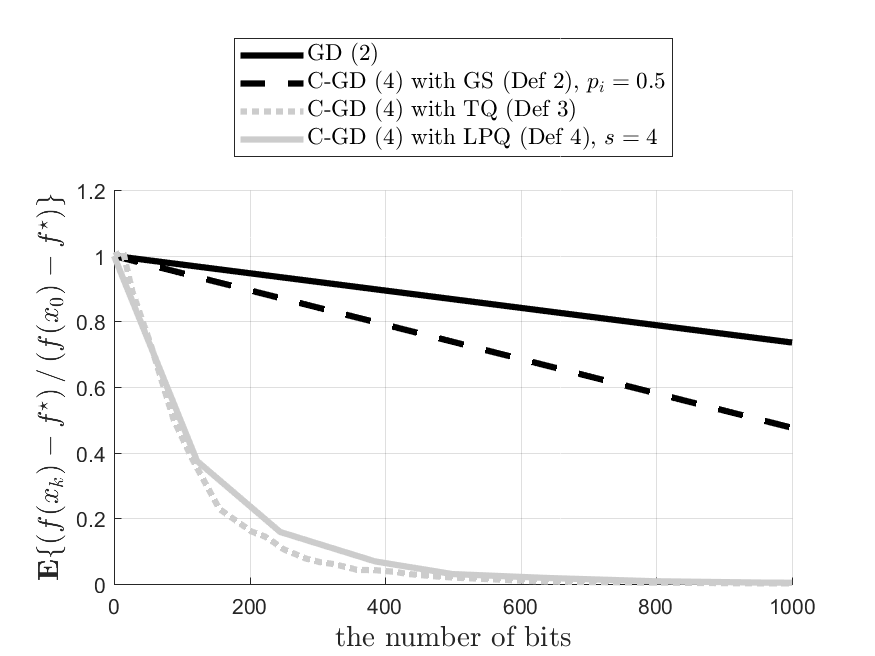}
\caption{}
\label{fig:CG_CoordCovtype}
\end{subfigure}
\caption{Convergence of compressed gradient descent algorithms \eqref{eqn:QuantizedGD} using different compression techniques over real-world data sets; that is, (a) real-sim, (b) RCV1-train and (c) covtype.}
\label{fig:CGCoordinates}
\end{figure}

\begin{figure}
\begin{subfigure}{.5\linewidth}
\centering
\includegraphics[scale = 0.44]{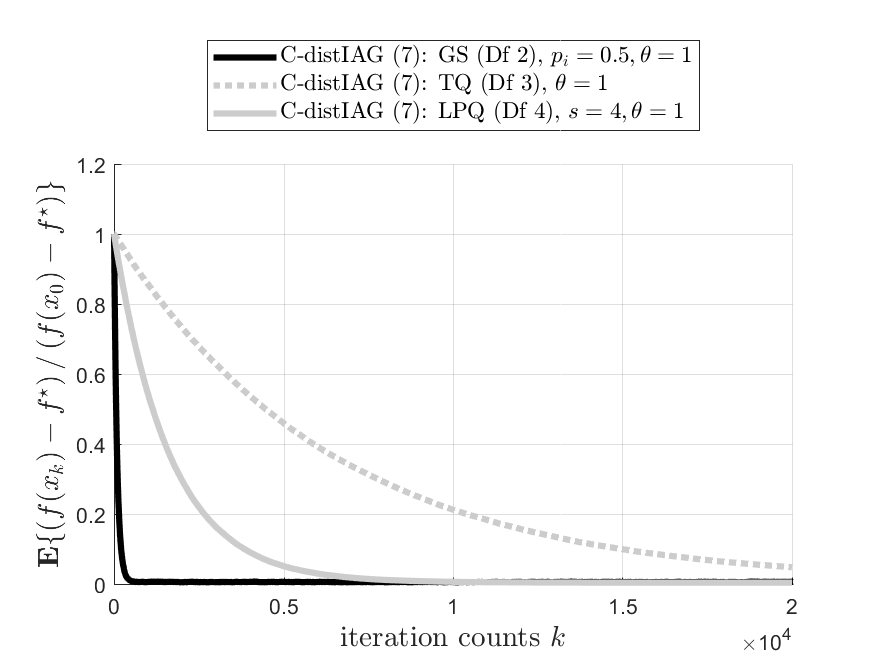}
\caption{}
\label{fig:distCIAG_IterRealsim}
\end{subfigure}%
\begin{subfigure}{.5\linewidth}
\centering
\includegraphics[scale = 0.44]{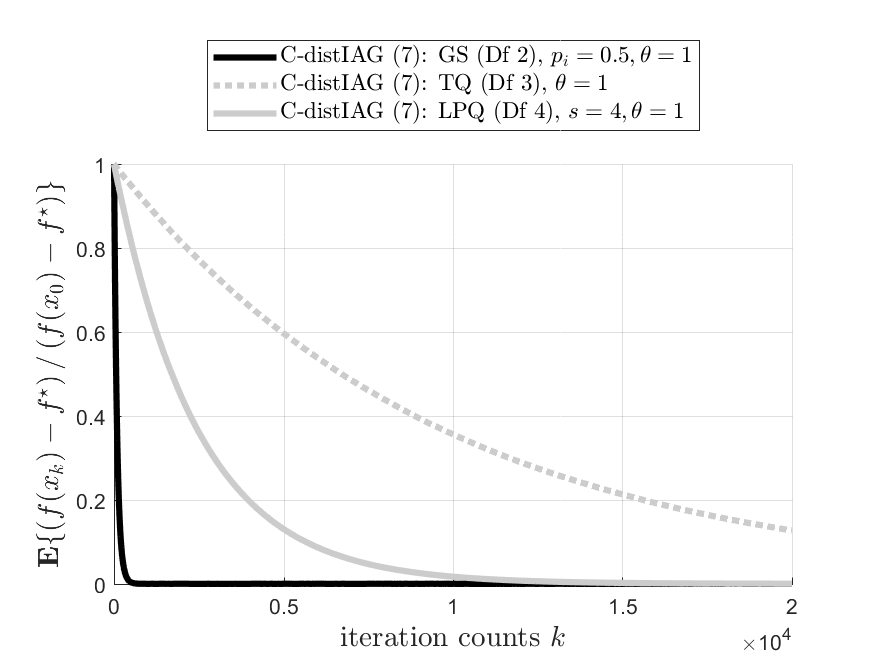}
\caption{}
\label{fig:distCIAG_IterRCV1}
\end{subfigure}\\[1ex]
\begin{subfigure}{\linewidth}
\centering
\includegraphics[scale = 0.44]{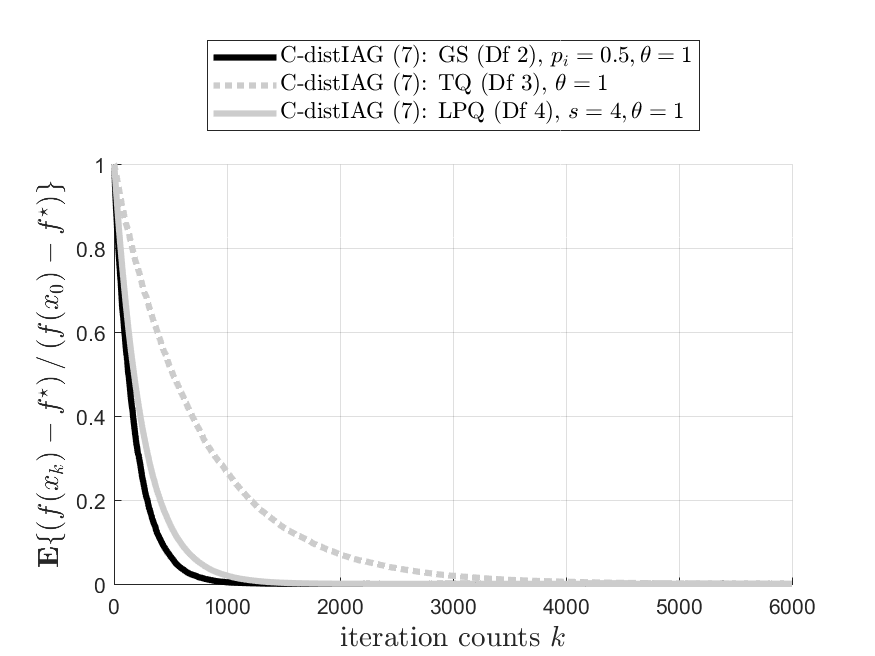}
\caption{}
\label{fig:distCIAG_IterCovtype}
\end{subfigure}
\caption{Convergence of Q-IAG algorithms \eqref{eqn:QIAG} using different compression techniques over real-world data sets; that is, (a) real-sim, (b) RCV1-train and (c) covtype. }
\label{fig:distCIAGIter}
\end{figure}

\begin{figure}
\begin{subfigure}{.5\linewidth}
\centering
\includegraphics[scale = 0.44]{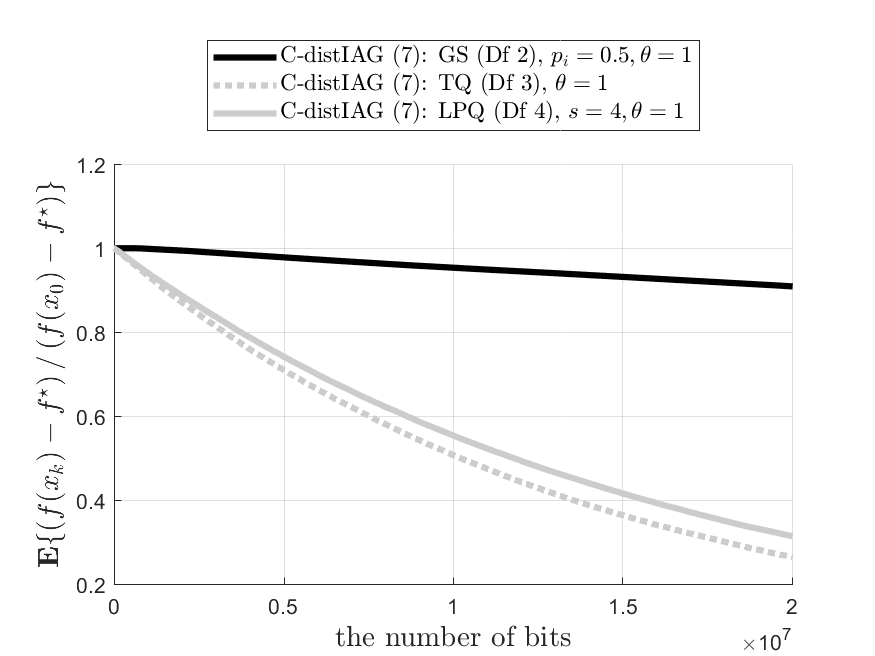}
\caption{}
\label{fig:distCIAG_CoordRealsim}
\end{subfigure}%
\begin{subfigure}{.5\linewidth}
\centering
\includegraphics[scale = 0.44]{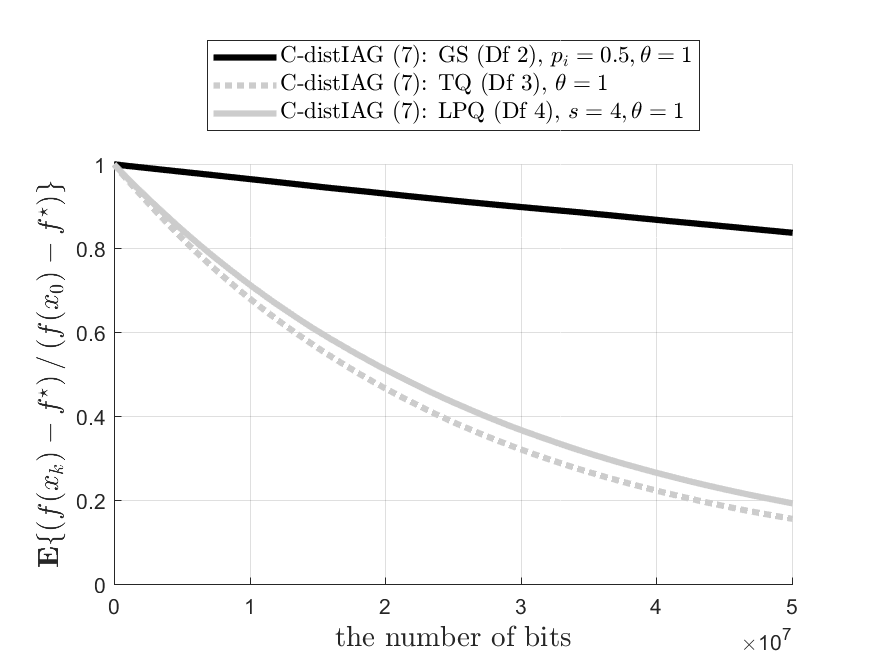}
\caption{}
\label{fig:distCIAG_CoordRCV1}
\end{subfigure}\\[1ex]
\begin{subfigure}{\linewidth}
\centering
\includegraphics[scale = 0.44]{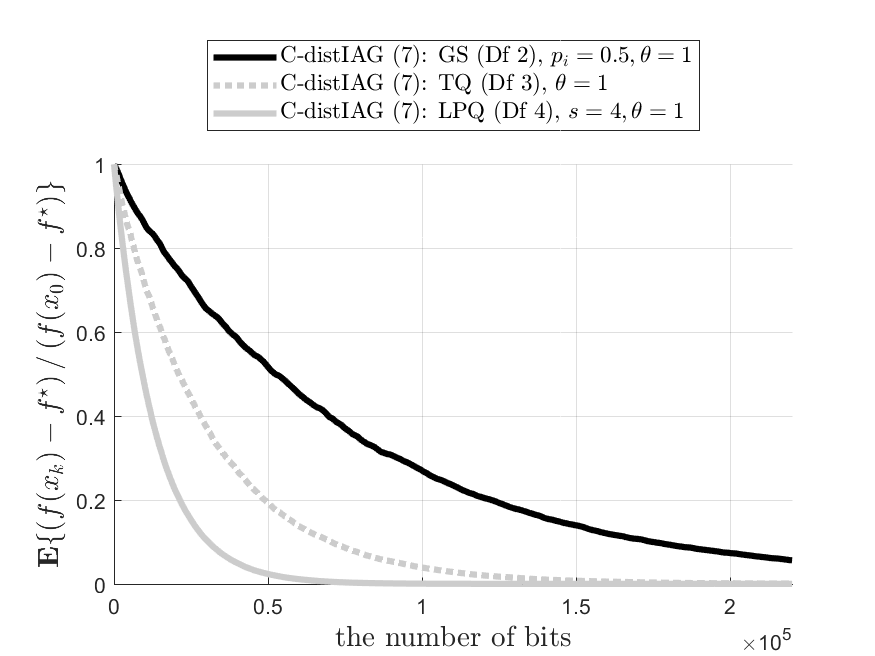}
\caption{}
\label{fig:distCIAG_CoordCovtype}
\end{subfigure}
\caption{Convergence of Q-IAG algorithms \eqref{eqn:QuantizedGD} using different compression techniques over real-world data sets; that is, (a) real-sim, (b) RCV1-train and (c) covtype.}
\label{fig:distCIAGCoordinates}
\end{figure}

   \section{Conclusions and Future Work}
    
    We have established a unified framework for both synchronous and asynchronous distributed optimization using compressed gradients. The framework builds on the concept of unbiased randomized quantizers (URQs), a class of gradient compression schemes which cover several important proposals from the literature~ \cite{wangni2017gradient,alistarh2016qsgd}. We have established non-asymptotic convergence rate guarantees for both GD and IAG with URQ compression. The convergence rate guarantees give explicit formulas for how quantization accuracy and staleness bounds affect the expected time to reach an $\varepsilon$-optimal solution. These results allowed us to characterize the trade-off between iteration and communication complexity of gradient descent under gradient compression. 
    
	We are currently working on extending the framework to allow for deterministic quantizers. Such quantizers are not necessarily unbiased, but satisfy additional inequalities which could be useful for the analysis. Another research direction is to establish non-asymptotic convergence rates under quantization-error compensation, which have been reported to work well in empirical studies~\cite{1-bit-stochastic-gradient-descents}. Finally, we would also like to analyze the effect of compressing the traffic from master to workers.

\appendix

%=================================================================================
%
% Lemma: Tighten proof of tightened \bar L 
%
%=================================================================================

\section{Proof of Lemma \ref{lemma:LipschitzWholeFcn}}  \label{app:lemma:LipschitzWholeFcn}
    For $x,y \in \mathbb{R}^d$, 
\({\left\| {x + y} \right\|^2} = {\left\| x \right\|^2} + {\left\| y \right\|^2} + 2\left\langle {x,y} \right\rangle \).  Together with the fact that $f(x)=\sum_{i=1}^m f_i(x)$, this property implies that  
\[
\begin{array}{l}
{\left\| {\nabla f(x) - \nabla f(y)} \right\|^2}  \\
%= {\left\| {\sum\limits_{i = 1}^m {\nabla {f_i}(x)}  - \sum\limits_{i = 1}^m {\nabla {f_i}(y)} } \right\|^2}\\
%
%
%
%
= {\left\| {\nabla {f_1}(x) - \nabla {f_1}(y)} \right\|^2} + \sum\limits_{i = 2}^m {\left\langle {\nabla {f_1}(x) - \nabla {f_1}(y),\nabla {f_i}(x) - \nabla {f_i}(y)} \right\rangle } \\
\hspace{0.5cm} + {\left\| {\nabla {f_2}(x) - \nabla {f_2}(y)} \right\|^2} + \sum\limits_{i = 1,i \ne 2}^m {\left\langle {\nabla {f_2}(x) - \nabla {f_2}(y),\nabla {f_i}(x) - \nabla {f_i}(y)} \right\rangle }  + ...\\
\hspace{0.5cm} + {\left\| {\nabla {f_m}(x) - \nabla {f_m}(y)} \right\|^2} + \sum\limits_{i = 1}^{m - 1} {\left\langle {\nabla {f_m}(x) - \nabla {f_m}(y),\nabla {f_i}(x) - \nabla {f_i}(y)} \right\rangle } \\

= \sum\limits_{i = 1}^m {{{\left\| {\nabla {f_i}(x) - \nabla {f_i}(y)} \right\|}^2}}  + \sum\limits_{i = 1}^m {\sum\limits_{j = 1,j \ne i}^m {\left\langle {\nabla {f_i}(x) - \nabla {f_i}(y),\nabla {f_j}(x) - \nabla {f_j}(y)} \right\rangle } } {e_{i,j}},
\end{array}\]
where $e_{i,j}=1$ if $\textup{supp}(\nabla f_i(x)) \cap \textup{supp}(\nabla f_j(x)) \neq \emptyset$ and $0$ otherwise. By Cauchy-Schwarz's inequality, we have
\begin{align*}
{\left\| {\nabla f(x) - \nabla f(y)} \right\|^2} &\leq \sum\limits_{i = 1}^m {{{\left\| {\nabla {f_i}(x) - \nabla {f_i}(y)} \right\|}^2}}  \\
&\hspace{0.5cm}+ \sum\limits_{i = 1}^m {\sum\limits_{j = 1,j \ne i}^m {\left\| {\nabla {f_i}(x) - \nabla {f_i}(y)} \right\| \cdot \left\| {\nabla {f_j}(x) - \nabla {f_j}(y)} \right\|} } {e_{i,j}}. 
\end{align*}
The Lipschitz continuity of the gradients of component functions $f_i$  implies that
\begin{align*}
{\left\| {\nabla f(x) - \nabla f(y)} \right\|^2} 
%& \leq \sum\limits_{i = 1}^m {L^2{{\left\| {x - y} \right\|}^2}}  + \sum\limits_{i = 1}^m {\sum\limits_{j = 1,j \ne i}^m {{L^2}{}\left\| {x - y} \right\| \cdot \left\| {x - y} \right\|} } {e_{i,j}}\\
%
%
%&= \left( {\sum\limits_{i = 1}^m {L^2}  + \sum\limits_{i = 1}^m {\sum\limits_{j = 1,j \ne i}^m {{L^2}{}{e_{i,j}}} } } \right){\left\| {x - y} \right\|^2} \\ 
%
% 
&\leq L_{ }^2\left( {m + \sum\limits_{i = 1}^m {\sum\limits_{j = 1,j \ne i}^m {{e_{i,j}}} } } \right){\left\| {x - y} \right\|^2} \\ 
&= L_{}^2m\left( {1 + {\Delta_{\rm ave}}} \right){\left\| {x - y} \right\|^2},
\end{align*}
Notice that the sparsity pattern of $\nabla f_i(x)$ can be found using the data matrix $A$, \cite{Sarit2018}.

Next, we can tighten the bound using the maximum conflict degree $\Delta_{\max}$. By Cauchy-Schwarz's inequality and by the Lipschitz gradient assumption of $f_i$, we have
\begin{align*}
    {\left\| {\nabla f(x) - \nabla f(y)} \right\|^2} 
 & \leq L_{ }^2\left( {m + \sum\limits_{i = 1}^m {\sum\limits_{j = 1,j \ne i}^m {{e_{i,j}}} } } \right){\left\| {x - y} \right\|^2} \\ 
  & \leq L_{ }^2 m\left( { 1 + \Delta_{\max}  }  \right){\left\| {x - y} \right\|^2},
\end{align*}
where the last inequality derives from the definition of the maximum conflict graph degree. In conclusion, 
\begin{align*}
    \bar{L}^2 & =  L^2m\left( {1 + \Delta}\right),
\end{align*}
where $\Delta = {\min}(\Delta_{\rm ave},\Delta_{\max}).$
%
%

  %=========================================================================================
    %
    % Proof theorem: Q-GD for strongly convex opt 
    %
    %=========================================================================================
    \section{Proof of Theorem \ref{thm:SCSerialQGD}}  \label{app:thm:SCSerialQGD}
    
Using the distance between the iterates $\{x_k\}_{k\in \mathbb{N}}$  and the optimum $x^\star$, we have 
    \begin{align*}
        \|x_{k+1} - x^\star \|^2 & = \| x_k - x^\star \|^2 - 2\gamma_k \langle Q(\nabla f(x_k)),x_k-x^\star  \rangle + \gamma_k^2 \| Q(\nabla f(x_k)) \|^2.
    \end{align*}
Taking the expectation with respect to all the randomness in the algorithm yields
    \begin{align*}
        \mathbf{E}\|x_{k+1} - x^\star \|^2 & = \mathbf{E}\| x_k - x^\star \|^2 - 2\gamma_k \mathbf{E}\langle \nabla f(x_k),x_k-x^\star  \rangle + \gamma_k^2 \mathbf{E}\| Q(\nabla f(x_k)) \|^2\\
        & \leq \mathbf{E}\| x_k - x^\star \|^2 - 2\gamma_k \mathbf{E}\langle \nabla f(x_k),x_k-x^\star  \rangle + \gamma_k^2\alpha \mathbf{E}\| \nabla f(x_k) \|^2,
    \end{align*}
    where the inequality follows from the second property in Definition~\ref{def:UnbiasedRandQuant}. Denote $V_k = \mathbf{E}\|x_{k+1} - x^\star \|^2$. It follows from~\cite[Theorem 2.1.12]{nesterov2013introductory} that
    \begin{align*}
      V_{k+1}  &\leq  \left( 1 - 2\gamma_k\frac{\mu \bar L}{\mu + \bar L}  \right)V_k + \left( -2\gamma_k\frac{1}{\mu+\bar L} + \gamma_k^2 \alpha \right)\mathbf{E}\| \nabla f(x_k) \|^2\\
        & = \rho_k V_k + \left( -2\gamma_k\frac{1}{\mu+\bar L} + \gamma_k^2 \alpha \right)\mathbf{E}\| \nabla f(x_k) \|^2,
    \end{align*}
    where $\rho_k = 1-2\gamma_k\frac{\mu \bar L}{\mu +\bar L}$. If $-2\gamma_k/{(\mu+\bar L)} + \gamma_k^2 \alpha \leq 0$, or equivalently 
    \begin{align*}
    \gamma_k \in \left(0, \frac{2}{\alpha(\mu+\bar L)}\right],
    \end{align*}
    then $\rho_k \in [0,1)$ for $\alpha \geq 1$, and the second term on the right-hand side of the above inequality is non-positive. Therefore, $V_{k+1}\leq \rho_k V_k$, which implies that \( V_k \leq  \left(\prod_{i=1}^k \rho_i \right) V_0,\) for all \( k\in\mathbb{N}. \) If the step-sizes are constant ($\gamma_k=\gamma$ for all $k\in\mathbb{N}$), then \( V_k \leq \rho^k V_0,\) for \( k\in\mathbb{N}\).

    %==========================================================================
    %
    % Corollary Q-GD for strongly convex opt
    %
    %==========================================================================
    \section{Proof of Corollary \ref{corr:SCSerialQGD}} \label{app:corr:SCSerialQGD}
     From Theorem \ref{thm:SCSerialQGD}, we get $V_k \leq \rho^k \varepsilon_0$ with $V_k = \mathbf{E}\Vert x_k -x^\star \Vert^2$, or equivalently 
    \begin{align*}
        \left( 1-\frac{1}{\alpha}\frac{4\mu\cdot  \bar L}{(\mu+\bar L)^2}  \right)^k\varepsilon_0 \leq \varepsilon. 
    \end{align*}
    Since $-1/{\rm log}(1-x) \leq 1/x$ for $0<x\leq 1$ and $\rho\in (0,1)$, we reach the upper bound of $k^\star$. In addition, assume that the number of non-zero elements is at most $c$. Therefore, the number of bits required to code the vector  is at most $\left( {\rm log}_2d + B \right)c$ bits in each iteration, where $B$ is the number bits required to encode a single vector entry. Hence, we reach the upper bound of $B^\star$.

  %=========================================================================================
    %
    % Proof theorem: Q-GD for convex opt 
    %
    %=========================================================================================
    \section{Proof of Theorem \ref{thm:QGDConvex}} \label{app:thm:QGDConvex}
    Denote $V_k =  \mathbf{E}\| x_k - x^\star \|^2$ and $\bar L =L\sqrt{m(1+\Delta)}$. Following the proof in Theorem \ref{thm:SCSerialQGD}, we have
    \begin{align*}
        V_{k+1}& \leq V_k - 2\gamma_k \mathbf{E}\langle \nabla f(x_k),x_k-x^\star  \rangle + \gamma_k^2\alpha \mathbf{E}\| \nabla f(x_k) \|^2. 
    \end{align*}    
    By  the property of Lipschitz continuity of $\nabla f(x)$, we have:  
    \begin{align*}
       \langle \nabla f(x_k),x_k-x^\star  \rangle & \geq   f(x_k) - f^\star + \frac{1}{2\bar L}\| \nabla f(x_k) \|^2,
    \end{align*}   
and by assuming that  $\gamma_k \leq 1/(\bar L\alpha)$, we get
\begin{align*}
 V_{k+1} \leq V_k - 2\gamma_k \mathbf{E}\left( f(x_k) - f^\star \right).
\end{align*}
    After the manipulation, we have: 
    \begin{equation}
        2 \sum_{k=0}^{T} \gamma_k \mathbf{E}\left( f(x_k) - f^\star \right) \leq V_0 - V_{T+1}. 
        \label{eqn:summaryQGDConvex}
    \end{equation}
    Again from the Lipschitz gradient assumption of $f$, we have
    \begin{align*}
        f(x_{k+1}) \leq f(x_k) - \gamma_k \langle \nabla f(x_k) , Q(\nabla f(x_k)) \rangle + \frac{\bar L\gamma_k^2}{2}\| Q(\nabla f(x_k)) \|^2. 
    \end{align*}
    Taking the expectation over all random variables yields
    \begin{align*}
       \mathbf{ E}f(x_{k+1}) \leq \mathbf{E} f(x_k)  - \left(\gamma_k - \frac{\bar L\alpha\gamma_k^2}{2} \right) \| \nabla f(x_k) \|^2,
    \end{align*}
    where we reach the inequality by properties stated in  Definition \ref{def:UnbiasedRandQuant}.  Due to the fact that $\gamma_k \leq 1/(\bar L\alpha)$ and the non-negativity of the Euclidean norm, we can conclude that $ \mathbf{ E}f(x_{k+1}) \leq \mathbf{E} f(x_k)$. From \eqref{eqn:summaryQGDConvex}, 
    \begin{align*}
    2\gamma_{\min} (T+1) \mathbf{E}\left( f(x_T) - f^\star \right) \leq \mathbf{E}\| x_0 - x^\star \|^2 - \mathbf{E}\|x_{T+1} - x^\star \|^2, 
    \end{align*}
    or equivalently 
    \begin{align*}
        \mathbf{E}\left( f(x_T) - f^\star \right) \leq \frac{1}{2\gamma_{\min}(T+1)}\mathbf{E}\| x_0 - x^\star \|^2,
    \end{align*}
    where $\gamma_{\min} = \mathop{\min}\limits_{k\in[0,T]} \gamma_k.$ Plugging $\gamma_{\min} = {1}/{(\bar L \alpha)}$  yields the result.

  %=========================================================================================
    %
    % Proof theorem: Q-GD for convex opt 
    %
    %=========================================================================================
    \section{Complexity of Compressed GD Algorithm for Convex Optimization} \label{app:corr:QGDConvex}
  The upper bound of $T^\star$ is easily obtained by using the inequality in Theorem \ref{thm:QGDConvex}. Also,  assume that the number of non-zero elements is at most $c$. Therefore, the number of bits required to code the vector is at most $\left( {\rm log}_2d + B \right)c$ bits in each iteration, where $B$ is the number bits required to encode a single vector entry. Hence, we reach the upper bound of $B^\star$.

  %=========================================================================================
    %
    % Proof lemma: sparsity trick 
    %
    %=========================================================================================
    \section{Proof of Lemma \ref{lemma:datasparsityTrick}} \label{app:lemma:datasparsityTrick}
    Denote $s_i^k = {\rm supp}(Q(a_i))$. By Assumption \ref{assum:SparsityEmpiricalLoss} and the definition of the Euclidean norm, 
    \begin{align*}
    \left\|  \sum_{i=1}^m Q(\nabla f_i(x_k)) \right\|^2  
   % & = \|Q(\nabla f_1(x_k)) \|^2 + \sum_{j=2}^m \langle  Q(\nabla f_1(x_k)), Q(\nabla f_j(x_k))\rangle  \\ 
   % & \hspace{0.4cm} + \|Q(\nabla f_2(x_k)) \|^2 + \sum_{j=1,j\neq 2}^m \langle  Q(\nabla f_2(x_k)), Q(\nabla f_j(x_k))\rangle + \ldots \\
%    & \hspace{0.4cm} + \|Q(\nabla f_m(x_k)) \|^2 + \sum_{j=1}^{m-1} \langle  Q(\nabla f_m(x_k)), Q(\nabla f_j(x_k))\rangle \\ 
    & = \sum_{i=1}^m \|Q(\nabla f_i(x_k)) \|^2 + \sum_{i=1}^m   \sum_{j=1,j\neq i}^{m} \langle  Q(\nabla f_i(x_k)), Q(\nabla f_j(x_k))\rangle \\
    & \leq \sum_{i=1}^m \|Q(\nabla f_i(x_k)) \|^2  + T,\intertext{where}
   T & =\sum_{i=1}^m   \sum_{j=1,j\neq i}^{m} \| Q(\nabla f_i(x_k))\| \| Q(\nabla f_j(x_k))\| \mathbf{1}\left( s_i^k \cap s_j^k \neq \emptyset  \right).%\\ 
    %& \leq \sum_{i=1}^m \|Q(\nabla f_i(x_k)) \|^2 \\
    %&\hspace{0.4cm}+ \sum_{i=1}^m   \sum_{j=1,j\neq i}^{m} \left( \frac{1}{2}\| Q(\nabla f_i(x_k))\|^2 + \frac{1}{2}\| Q(\nabla f_j(x_k))\|^2  \right)\mathbf{1}\left( s_i^k \cap s_j^k \neq \emptyset  \right) \\ 
    %& \leq \left( 1+ \Delta_{\max}^k \right) \sum_{i=1}^m \|Q(\nabla f_i(x_k)) \|^2,
    \end{align*}
    We reach the inequality by Cauchy-Schwarz's inequality. For simplicity, let $e_{i,j} = \mathbf{1}\left( s_i^k \cap s_j^k \neq \emptyset  \right)$ and  $\Delta_i = \sum_{j=1,j\neq i}^m e_{i,j}.$ Therefore, we define the maximum conflict degree $\Delta_{\max}^k = \max_{i\in[1,m]}\Delta_i $ and the average conflict degree $\Delta_{\rm ave}^k = (\sum_{i=1}^m \Delta_i)/m$. Now, we bound the left-hand side by using two different data sparsity measures. First, we bound $T$ by using the maximum conflict degree $\Delta_{\max}^k$. By the fact that $2ab\leq a^2+b^2$ for $a,b\in\mathbb{R}$, we have 
   \begin{align*}
       T & \leq \frac{1}{2}\sum_{i=1}^m   \sum_{j=1,j\neq i}^{m} \left( \| Q(\nabla f_i(x_k))\|^2 + \| Q(\nabla f_j(x_k))\|^2 \right)e_{i,j}  \\
&\leq \Delta_{\max}^k \sum_{i=1}^m \| Q(\nabla f_i(x_k))\|^2.
   \end{align*}
   Therefore, 
     \begin{align*}
      \left\|  \sum_{i=1}^m Q(\nabla f_i(x_k)) \right\|^2 \leq (1+\Delta_{\max}^k)\sum_{i=1}^m \| Q(\nabla f_i(x_k))\|^2.
      \end{align*}
    Next, we bound the left-hand side by using the average conflict degree $\Delta_{\rm ave}^k$. By Cauchy-Schwarz's inequality, we get: 
    \begin{align*}
      \left\|  \sum_{i=1}^m Q(\nabla f_i(x_k)) \right\|^2   & = \sum_{i=1}^m   \| Q(\nabla f_i(x_k)\| \sum_{j=1}^{m} \| Q(\nabla f_j(x_k)) \| e_{i,j} \\
       & \leq  \sum_{i=1}^m \| Q(\nabla f_i(x_k))\| \sqrt{\sum_{j=1}^{m} \| Q(\nabla f_j(x_k))\|^2 \sum_{j=1}^{m}  e^2_{i,j}} \\ 
       & \leq \sqrt{\sum_{i=1}^m \| Q(\nabla f_i(x_k))\|^2 } \sqrt{\sum_{j=1}^{m} \| Q(\nabla f_j(x_k))\|^2} \sqrt{\sum_{i=1}^m\sum_{j=1}^{m}  e^2_{i,j}} \\ 
       & \leq \sqrt{\sum_{i=1}^m\sum_{j=1}^{m}  e_{i,j}}\sum_{i=1}^m \| Q( \nabla f_i(x_k))\|^2 \\
       & = \sqrt{m(1+\Delta_{\rm ave}^k)}\sum_{i=1}^m \| Q(\nabla f_i(x_k))\|^2.
   \end{align*}
    In conclusion, 
    \begin{align*}
         \left\|  \sum_{i=1}^m Q(\nabla f_i(x_k)) \right\|^2   \leq \sigma_k \sum_{i=1}^m \| Q(\nabla f_i(x_k))\|^2,
    \end{align*}
    where $\sigma_k = {\min}\left( \sqrt{m(1+\Delta_{\rm ave}^k)},1+\Delta_{\max}^k \right)$.
    
      %==========================================================================================
    %
    % Proof serial Quantized IAG for strongly convex optimization 
    %
    %==========================================================================================
    \section{Proof of Theorem \ref{thm:QthenIAGSCTrickHamid}} \label{app:thm:QthenIAGSCTrickHamid}
    
    Denote  $g_k = \sum_{i=1}^m \nabla f_i(x_{k-\tau_k^i}).$  Let us first introduce two main lemmas which are instrumental to our main analysis. 
        
    \begin{lemma}
    \label{lemma:errorQIAGnondist2}
    Consider the iterates generated by~\eqref{eqn:QIAGnonDistr}. For $k\in\mathbb{N}_0$, 
    \begin{align*}
        \| g_k  \|^2  \leq \frac{2 {\bar L}^2 }{\mu} \mathop{\max}\limits_{s\in[k-\tau,k]} f(x_{s}) -f(x^\star),
     \end{align*}
    where  $\bar L =L\sqrt{m(1+\Delta)}$ and $\Delta = {\min}(\Delta_{\rm ave},\Delta_{\max})$.
    \end{lemma}
    
    \begin{proof}
    Since $\nabla f(x^\star) = 0$, we have
    \begin{align*}
        \| g_k \|^2  = \left\|\sum_{i=1}^m  \nabla f_i(x_{k-\tau_k^i}) -\nabla f_i(x^\star) \right\|^2. 
\end{align*}
Following the proof of Lemma \ref{lemma:LipschitzWholeFcn} with $x=x_{k-\tau_k^i}$ and $y=x^\star$ yields 
\begin{align*}
     \| g_k \|^2    
        & \leq {\bar L}^2 \mathop{\max}\limits_{s\in[k-\tau,k]} \| x_{s} -x^\star \|^2,
    \end{align*}
    where $\bar L =L\sqrt{m(1+\Delta)}$ and $\Delta = {\min}(\Delta_{\rm ave},\Delta_{\max})$. The result also uses the fact that \[\|x_{k-\tau_k^i}-x^\star\| \leq \mathop{\max}\limits_{s\in[k-\tau,k]} \| x_{s} -x^\star \|.\] Since
    \begin{align*}
    f(x) - f(x^\star)\geq \frac{\mu}{2} \| x -x^\star \|^2,
   \end{align*}
    for any $x$, it follows that
    \begin{align*}
        \| g_k \|^2\leq \frac{2 {\bar L}^2 }{\mu} \mathop{\max}\limits_{s\in[k-\tau,k]} f(x_{s}) -f(x^\star).
    \end{align*}
    \end{proof}
    
    \begin{lemma}
    \label{lemma:errorQIAGnondist1}
    The sequence $\{x_k\}$ generated by~\eqref{eqn:QIAGnonDistr} satisfies 
    \begin{align*}
               \mathbf{E} \| \nabla f(x_k) - g_k \|^2  \leq \frac{2\gamma^2{\bar L}^4 \tau^2 \alpha}{\mu} \mathop{\max}\limits_{s\in[k-2\tau,k]} f(x_{s}) -f(x^\star),
    \end{align*}
    for $k\in\mathbb{N}_0$, where $\bar L =L\sqrt{m(1+\Delta)}$ and $\Delta = {\min}(\Delta_{\rm ave},\Delta_{\max})$.
    \end{lemma}
    \begin{proof}
    By the definition of $g_k$, 
        \begin{align*}
        \| \nabla f(x_k) - g_k \|^2 & = \left  \|\sum_{i=1}^m \nabla f_i(x_k) - \nabla f_i(x_{k-\tau_k^i})\right\|^2.
        \end{align*}
    Following the proof of Lemma \ref{lemma:LipschitzWholeFcn} with $x=x_{k}$ and $y=x_{k-\tau_k^i}$ yields 
      \begin{align*}
        \| \nabla f(x_k) - g_k \|^2 
        & \leq \bar L^2 \mathop{\max}\limits_{i\in[1,m]}\left\| x_k - x_{k-\tau_k^i}\right\|^2, \\ 
        \end{align*}
  where $\bar L =L\sqrt{m(1+\Delta)}$ and $\Delta= {\min}(\Delta_{\rm ave},\Delta_{\max})$. We also reach the result by the fact that \[\left\| x_k - x_{k-\tau_k^i}\right\| \leq \mathop{\max}\limits_{i\in[1,m]}\left\| x_k - x_{k-\tau_k^i}\right\|.\] 
Next, notice that 
\begin{align*}
      \| \nabla f(x_k) - g_k \|^2     & \leq \bar L^2 \mathop{\max}\limits_{i\in[1,m]}\left\| \sum_{j = k-\tau_k^i}^{k-1} x_{j+1} - x_j\right\|^2 \\ 
      & \leq \bar L^2  \mathop{\max}\limits_{i\in[1,m]} \tau_k^i \sum_{j = k-\tau_k^i}^{k-1}  \left \|x_{j+1} - x_j\right\|^2\\
      & \leq \bar L^2 \tau \sum_{j = k-\tau}^{k-1}  \left \|x_{j+1} - x_j\right\|^2 \\
      & = \bar L^2 \gamma^2\tau  \sum_{j = k-\tau}^{k-1}  \left \|Q(g_j)\right\|^2.
    \end{align*}
   The second inequality derives from the bounded delay assumption. Taking the expectation with respect to the randomness yields
     \begin{align*}   
          \mathbf{E}\| \nabla f(x_k) - g_k \|^2\leq  \gamma^2 {\bar L}^2 \tau \alpha  \sum_{j = k-\tau}^{k-1}  \left \|g_j\right\|^2.
     \end{align*}
     It follows from Lemma~\ref{lemma:errorQIAGnondist2} that
     \begin{align*}
      \mathbf{E}\| \nabla f(x_k) - g_k \|^2 &\leq \frac{2\gamma^2{\bar L}^4 \tau \alpha}{\mu} \sum_{j = k-\tau}^{k-1} \mathop{\max}\limits_{s\in[j-\tau,j]} f(x_{s}) -f(x^\star)\\
      & \leq \frac{2\gamma^2{\bar L}^4 \tau^2 \alpha}{\mu} \mathop{\max}\limits_{s\in[k-2\tau,k]} f(x_{s}) -f(x^\star).
     \end{align*}    
    \end{proof}
    
  We now prove Theorem \ref{thm:QthenIAGSCTrickHamid}. Since the entire cost function $f$ has Lipschitz continuous gradient with constant $\bar L$, we have
  \begin{align*}
  f(x_{k+1}) - f(x^\star)  \leq f(x_k)- f(x^\star) - \gamma \langle Q(g_k), \nabla f(x_k) \rangle + \frac{\gamma^2 \bar L}{2}\| Q(g_k) \|^2.
  \end{align*}
  Taking the expectation with respect to the randomness and using the second property in Definition~\ref{def:UnbiasedRandQuant}, we obtain
  \begin{align*}
  \mathbf{E} [f(x_{k+1})- f(x^\star)] \leq \mathbf{E}[f(x_k)- f(x^\star)] - \gamma\mathbf{E} [\langle g_k, \nabla f(x_k) \rangle] + \frac{\gamma^2 \alpha\bar L}{2}\mathbf{E} [\| g_k \|^2].
  \end{align*}
  If $\gamma \alpha \bar{L}\leq 1$, then $\gamma^2 \alpha\bar L \leq \gamma$, which implies that
  \begin{align*}
  \mathbf{E} [f(x_{k+1})- f(x^\star)] \leq \mathbf{E}[f(x_k)- f(x^\star)] - \gamma\mathbf{E} [\langle g_k, \nabla f(x_k) \rangle] + \frac{\gamma}{2}\mathbf{E} [\| g_k \|^2].
  \end{align*}
  Using $g_k = g_k - \nabla f(x_k) + \nabla f(x_k)$, we have
  \begin{align*}
  \mathbf{E} [f(x_{k+1})- f(x^\star)] &\leq \mathbf{E}[f(x_k)- f(x^\star)] - \frac{\gamma}{2}\mathbf{E}[\| \nabla f(x_k)\|^2] + \frac{\gamma}{2}\mathbf{E} [\| g_k-\nabla f(x_k) \|^2]\\
  & \leq (1-\gamma \mu) \mathbf{E}[f(x_k)- f(x^\star)] + \frac{\gamma}{2}\mathbf{E} [\| g_k-\nabla f(x_k) \|^2],
  \end{align*}
  where the second inequality follows from the fact that
  \begin{align*}
  f(x) - f(x^\star) \leq \frac{1}{2\mu} \| \nabla f(x) \|^2,
  \end{align*}
  for any $x$. It follows from Lemma~\ref{lemma:errorQIAGnondist1} that
  \begin{align*}
  \mathbf{E} [f(x_{k+1})- f(x^\star)] & \leq (1-\gamma \mu) \mathbf{E}[f(x_k)- f(x^\star)] + \frac{\gamma^3{\bar L}^4 \tau^2 \alpha}{\mu} \mathop{\max}\limits_{s\in[k-2\tau,k]} f(x_{s}) -f(x^\star).
  \end{align*}
 This inequality can be rewritten as
 \begin{align*}
        V_{k+1} & \leq p V_k + q  \mathop{\max}\limits_{s\in [k-2\tau,k]} V_s, \\ \intertext{where}
        V_k & =  \mathbf{E}[f(x_k)- f(x^\star)] \\ 
        p & = 1- \gamma \mu \\ 
        q & =  \frac{\gamma^3{\bar L}^4 \tau^2 \alpha}{\mu}.
    \end{align*}
     According to Lemma 1 of \cite{aytekin2014asynchronous}, if $p+q < 1$, or, equivalently, 
     \begin{align*}
         \gamma < \frac{\mu}{{\bar L}^2\tau \sqrt{\alpha}},
     \end{align*}
 then $V_k \leq (p+q)^{k/(1+2\tau)} V_0$. This completes the proof.

   %==========================================================================================
    %
    % Corrolary serial Quantized IAG for strongly convex optimization 
    %
    %==========================================================================================
    \section{ Proof of Corollary \ref{corr:QthenIAGSCTrickHamid}}\label{app:corr:QthenIAGSCTrickHamid}
    From Theorem \ref{thm:QthenIAGSCTrickHamid}, we get $V_k \leq \rho^k \varepsilon_0$ with $V_k = \mathbf{E}\left( f(x_k) - f^\star \right)$, or equivalently 
    \begin{align*}
        \left(   1- \mu\gamma + \bar L^4 \gamma^3 \tau^2 \frac{\alpha}{\mu}\right)^{\frac{k}{1+2\tau}}\varepsilon_0 \leq \varepsilon. 
    \end{align*}
    Since $-1/{\rm log}(1-x) \leq 1/x$ for $0<x\leq 1$ and $\rho \in (0,1)$, we reach the upper bound of $k^\star$. In addition, assume that the number of non-zero elements is at most $c$. Therefore, the number of bits required to code the vector is at most $\left( {\rm log}_2d + B \right)c$ bits in each iteration, where $B$ is the number bits required to encode a single vector entry. Hence, we reach the upper bound of $B^\star$.

 %%%%%%%%%%%%%%%%%%%%%%%%%%%%%%%%%%%%%%%%%%%%%%%%%%%%%%%%%%%%%%%%%%%%%%%%%%
    %%%%%%%%%%%%%%%%%%%%%%%%%%%%%%%%%%%%%%%%%%%%%%%%%%%%%%%%%%%%%%%%%%%%%%%%%%
    % Compressed IAG non strongly convex 
    %%%%%%%%%%%%%%%%%%%%%%%%%%%%%%%%%%%%%%%%%%%%%%%%%%%%%%%%%%%%%%%%%%%%%%%%%%

    \section{Proof of Theorem \ref{thm:CompressedIAGNonstronglyConvex}  } \label{app:thm:CompressedIAGNonstronglyConvex}
     Define $g_k = \sum_{i=1}^m \nabla f_i(x_{k-\tau_k^i})$. Let us introduce three main lemmas which are instrumental in our main analysis. 
    
    \begin{lemma}
    \label{lemma:errorsquareNSC1}
    The sequence $\{x_k \}$ generated by \eqref{eqn:QIAGnonDistr} satisfies 
    \begin{align*}
      \| \nabla f(x_k) - g_k \|^2  \leq  \bar L^2 \gamma^2\tau  \sum_{j = k-\tau}^{k-1}  \left \|Q(g_j)\right\|^2,
    \end{align*}
     where $\bar L =  L\sqrt{m(1+\Delta)}$ and $\Delta = {\min}(\Delta_{\rm ave},\Delta_{\max})$.
    \end{lemma}
    \begin{proof}
    Following the proof in Lemma  \ref{lemma:errorQIAGnondist1} yields the result. 
    \end{proof}
    
    \begin{lemma}
    \label{lemma:errorsquareNSC2}
      The sequence $\{x_k \}$ generated by \eqref{eqn:QIAGnonDistr} satisfies 
     \begin{align*}
       \mathbf{E}\left\| \nabla f(x_k) - Q(g_k) \right\|^2 
         & \leq 2\left( 1 + \beta(1 + \theta)  \right)\mathbf{E}\left\| \nabla f(x_k) - g_k \right\|^2 \\
& \hspace{0.5cm}+ 2\beta(1 + 1/\theta)\mathbf{E}\| \nabla f(x_k)\|^2,
    \end{align*}
    where $\theta>0$. 
    \end{lemma}
    \begin{proof}
    We start by deriving the upper bound of $\mathbf{E}\left\|  g_k - Q(g_k) \right\|^2$. By the property stating that $\mathbf{E}\| Q(v) - v \|^2\leq\beta\|v\|^2$ and by the fact that $\nabla f(x^\star)=0$, we have: 
    \begin{align*}
        \mathbf{E}\left\|  g_k - Q(g_k) \right\|^2 & \leq \beta \| g_k \|^2 \\  
        & \leq \beta(1+\theta)\| g_k - \nabla f(x_k)\|^2 + \beta(1+1/\theta)\| \nabla f(x_k)\|^2,
    \end{align*}
    where the last inequality derives from the fact that $\|x+y\|^2\leq(1+\theta)\|x\|^2 + (1+1/\theta)\|y\|^2$ for $x,y\in\mathbb{R}^d$ and $\theta>0$.
    
    Now, we are ready to derive the upper bound of $\mathbf{E}\left\| \nabla f(x_k) - Q(g_k) \right\|^2$. By the fact that $\| \sum_{i=1}^N x_i\|^2\leq N\sum_{i=1}^N \| x_i \|^2$ for $x_i\in\mathbb{R}^d$ and $N\in \mathbb{N}$, we have 
    \begin{align*}
        \left\| \nabla f(x_k) - Q(g_k) \right\|^2 
        & \leq 2 \left\| \nabla f(x_k) - g_k \right\|^2 + 2\left\|  g_k - Q(g_k) \right\|^2.
    \end{align*}
    Taking the expectation over the randomness and then plugging the upper bound of $\mathbf{E}\left\|  g_k - Q(g_k) \right\|^2$ into the result  yield the result.
\iffalse
    \begin{align*}
       \mathbf{E}\left\| \nabla f(x_k) - Q(g_k) \right\|^2 
        & \leq 2 \mathbf{E}\left\| \nabla f(x_k) - g_k \right\|^2 + 2\mathbf{E}\left\|  g_k - Q(g_k) \right\|^2 \\ 
        & \leq 2\left( 1 + \beta(1 + \theta)  \right)\mathbf{E}\left\| \nabla f(x_k) - g_k \right\|^2 + 2\beta(1 + 1/\theta)\mathbf{E}\| \nabla f(x_k)\|^2.
    \end{align*}
\fi
    \end{proof}
    
    \begin{lemma}
\label{lemma:IAGConvegenceTrickResult}
Suppose that non-negative sequences $\{V_k\}, \{w_k\},$ and $\{\Theta_k\}$ satisfying the following inequality 
\begin{equation}
      V_{k+1} \leq V_k - a\Theta_k - b w_{k} + c \sum_{j=k-\tau}^k w_{j}, 
      \label{eqn:LemIAGDelayed}
\end{equation}
where $a,b,c >0$. Further suppose that $b-c(\tau + 1) \geq 0$ and $w_k =0 $  for $k < 0$. Then,  
\begin{align*}
     \frac{1}{K+1}\sum_{k=0}^{K}\Theta_k 
    & \leq \frac{1}{a}\frac{1}{K+1}( V_0 - V_{K+1} ).
\end{align*}
\end{lemma}
\begin{proof}
Summing \eqref{eqn:LemIAGDelayed} from $k=0$ to $k=K$ yields 
\begin{align*}
   \sum_{k=0}^{K} V_{k+1} & \leq \sum_{k=0}^K V_{k} - a\sum_{k=0}^K\Theta_k - b\sum_{k=0}^K w_k + c \sum_{k=0}^K \sum_{j=k-\tau}^k w_{j},
\end{align*}
or equivalently due to the telescopic series 
\begin{align*}
    a\sum_{k=0}^K\Theta_k & \leq  (V_0 - V_{K+1})- b\sum_{k=0}^K w_k + c \sum_{k=0}^K \sum_{j=k-\tau}^k w_{j} \\ 
    & =  (V_0 - V_{K+1})- b\sum_{k=0}^K w_k  \\ 
    & \hspace{0.5cm}  + c(w_{-\tau} + w_{-\tau+1}+\ldots + w_{0}) \\ 
    &  \hspace{0.5cm}   + c(w_{-\tau+1} + w_{-\tau+2}+\ldots + w_{0}+w_1) + \ldots \\ 
    &  \hspace{0.5cm}   + c(w_{-\tau+K} + w_{-\tau+K+1}+\ldots + w_0 + w_1 +\ldots + w_K) \\
    &  \leq (V_0 - V_{K+1})- b\sum_{k=0}^K w_k + c(\tau+1)\sum_{k=0}^K w_k \\
    & \leq V_0 - V_{K+1}, 
\end{align*}
where the second inequality comes from the fact that $w_k \geq 0$ for $k \geq 0.$ In addition, the last inequality follows from the assumption that $b-c(\tau+1) \geq 0$. Then, we obtain the result.
\end{proof}
    
    Now, we are ready to derive the convergence rate. From the definition of the Lipschitz continuity of the gradient of the function $f$, we have 
    \begin{align*}
        f(x_{k+1}) - f^\star& \leq f(x_{k}) - f^\star -\gamma \left \langle \nabla f(x_k), Q(g_k) \right\rangle + \frac{\gamma^2 \bar L}{2}\left\| Q(g_k) \right\|^2 \\
        & \leq f(x_k) - f^\star - \frac{\gamma}{2}\left\| \nabla f(x_k) \right\|^2 - \left( \frac{\gamma}{2} - \frac{\gamma^2\bar L}{2} \right)\left\| Q(g_k) \right\|^2 \\
        &\hspace{0.4cm} + \frac{\gamma}{2}\left\| \nabla f(x_k) - Q(g_k) \right\|^2,
    \end{align*}
    where $\bar L =  L\sqrt{m(1+\Delta)}$ and $\Delta = {\min}(\Delta_{\rm ave},\Delta_{\max})$. The last inequality derives from the fact that $2\langle x,y \rangle = \|x\|^2+\|y\|^2-\|x-y\|^2$ for any $x,y\in\mathbb{R}^d$. Denote  $V_k = \mathbf{E}f(x_k) - f^\star, \Theta_k = \mathbf{E}\|\nabla f(x_k)\|^2$, and $w_k = \mathbf{E}\|Q(g_k)\|^2$. Next, taking the expectation over the randomness, and then plugging the inequality from Lemma  \ref{lemma:errorsquareNSC1} and \ref{lemma:errorsquareNSC2} yield
    \begin{align*}
   V_{k+1}
        & \leq  V_k - \alpha_1 \Theta_k - \alpha_2 w_k + \alpha_3 \sum_{j = k-\tau}^{k-1} w_j, \intertext{where} 
        \alpha_1 & = \gamma/2 - \gamma\beta(1+1/\theta) \\ 
        \alpha_2 & =\gamma/2-\bar L\gamma^2/2\\
        \alpha_3 & = \gamma\left(1 + \beta(1+\theta) \right)\bar L^2 \gamma^2 \tau,
    \end{align*}
    and $\bar L =  L\sqrt{m(1+\Delta)}$ and $\Delta = {\min}(\Delta_{\rm ave},\Delta_{\max})$. Next, we apply Lemma \ref{lemma:IAGConvegenceTrickResult}. Notice that $\|Q(g_k)\|=\|x_{k+1}-x_k\|/\gamma$, which implies that $w_k=0$ if $k<0$. Therefore, 
    \begin{align*}
    \frac{1}{K+1} \sum_{k=0}^{K}\Theta_k 
    & \leq \frac{1}{a}\frac{1}{K+1}( V_0 - V_{K+1} ),
\end{align*}
which means that
  \begin{align*}
      \mathop{\min}\limits_{k\in[0,K]} \mathbf{E}\| \nabla f(x_k) \|^2 \leq \frac{1}{a}\frac{1}{K+1}\left( f(x_0) - f^\star\right) - \frac{1}{a}\frac{1}{K+1}\left( \mathbf{E} f(x_k) - f^\star\right) .
    \end{align*}
To ensure the validity of the result, we must determine $\gamma$ and $\beta$ to satisfy three conditions, \emph{i.e.} $a>0$, $b>0$ and $b-c(\tau+1)\geq 0$. The first criterion implies that $\beta < 1/\left(2(1+1/\theta)\right)$, and the last two criteria yield the admissible range of the step size $\gamma$. The second criterion implies that $\gamma<1/\bar L$, and the equivalence of the last criterion is 
\begin{align*}
    \frac{1}{2} - \frac{\bar L\gamma}{2} - \left(1 + \beta(1+\theta) \right)\bar L^2  \tau(\tau+1)\gamma^2 \geq 0. 
\end{align*}
Therefore, let $\gamma = \frac{1}{\bar L(1+\omega)}$ where $\omega > 0$, and plugging the expression into the inequality yields 
\begin{align*}
    \omega^2 + \omega - 2\psi \geq 0,
\end{align*}
where $\psi = \left(1 + \beta(1+\theta) \right)\tau(\tau+1).$ Therefore, $\omega \geq \left(-1 +\sqrt{1+8\psi} \right)/2$, and 
\begin{align*}
    \gamma < \frac{1}{\sqrt{1+8\left( 1 + \beta(1+\theta) \right)\tau(\tau+1)}}\frac{2}{\bar L}. 
\end{align*}

  %=========================================================================================
    %
    % Proof theorem: distributed Q-GD for strongly convex optimization 
    %
    %=========================================================================================
    \section{Proof of Theorem \ref{thm:DistributedQGDSC}} \label{app:thm:DistributedQGDSC}
    Since the component functions are convex and have $L$-Lipschitz continuous gradients,     \begin{equation}
    \|\nabla f_i(x) -\nabla f_i(y)\|^2 \leq L \langle 
    \nabla f_i(x)- \nabla f_i(y),x-y\rangle \qquad \forall x,y\in \mathbb{R}^d.
    \label{eqn:cocoercivity}
    \end{equation}
    By Young's inequality,
    \begin{equation}
    \begin{array}{rl}
    \|\nabla f_i(x_k)\|^2 &\leq (1+\theta)\| \nabla f_i(x_k) - \nabla f_i(x^\star) \|^2 + (1+1/\theta)\|\nabla f_i(x^\star)\|^2 \\
    & \leq (1+\theta)L\langle \nabla f_i(x_k) - \nabla f_i(x^\star), x_k - x^\star \rangle + (1+1/\theta)\|\nabla f_i(x^\star)\|^2
    \end{array}    
    \label{eqn:trick}
    \end{equation}
    
    We use the distance between the iterates $\{x_k\}_{k\in \mathbb{N}}$  and the optimum $x^\star$ to analyze the convergence: 
    \begin{align*}
    \|x_{k+1} - x^\star \|^2  
    &= \| x_k - x^\star \|^2 - 2\gamma_k \left\langle {\sum\limits_{i = 1}^m Q (\nabla {f_i}({x_k})),{x_k} - {x^ \star }} \right\rangle \\
    &\hspace{0.4cm}+ \gamma_k^2\left\|  \sum_{i=1}^m Q(\nabla f_i(x_k)) \right\|^2 \\ 
    & \leq \| x_k - x^\star \|^2 - 2\gamma_k \left\langle {\sum\limits_{i = 1}^m Q (\nabla {f_i}({x_k})),{x_k} - {x^ \star }} \right\rangle  \\
    &\hspace{0.4cm}+ \gamma_k^2 \sigma_k \sum_{i=1}^m \left\|   Q(\nabla f_i(x_k)) \right\|^2,
    \end{align*}
    where the second inequality comes from Lemma \ref{lemma:datasparsityTrick}. Notice that $\mathbf{E}\left\|   Q(\nabla f_i(x_k)) \right\|^2 \leq \alpha \mathbf{E}\| \nabla f_i(x_k) \|^2$, since all machines have the same quantizers with the same parameters. Therefore, taking the expectation over all random variables yields 
    \begin{align*}
    \mathbf{E}\|x_{k+1} - x^\star \|^2 
   % & = \mathbf{E}\| x_k - x^\star \|^2 - 2\gamma_k \mathbf{E}\left\langle \nabla f(x_k),{x_k} - {x^ \star } \right\rangle \\
   % &\hspace{0.4cm}+ \gamma_k^2 \sigma_k \sum_{i=1}^m \mathbf{E}\left\|   Q(\nabla f_i(x_k)) \right\|^2 \\ 
%    & \leq \mathbf{E}\| x_k - x^\star \|^2 - 2\gamma_k \mathbf{E}\left\langle \nabla f(x_k),{x_k} - {x^ \star } \right\rangle \\
  %  &\hspace{0.4cm}+ \gamma_k^2 \sigma_k  \sum_{i=1}^m \alpha \mathbf{E}\left\|   \nabla f_i(x_k) \right\|^2\\
    & \leq \mathbf{E}\| x_k - x^\star \|^2 - 2\gamma_k \mathbf{E}\left\langle \nabla f(x_k),{x_k} - {x^ \star } \right\rangle \\
    &\hspace{0.4cm}+ \gamma_k^2 \sigma_k  \alpha \sum_{i=1}^m  \mathbf{E}\left\|   \nabla f_i(x_k) \right\|^2 \\ 
    & \leq \mathbf{E}\| x_k - x^\star \|^2 - 2\gamma_k \mathbf{E}\left\langle \nabla f(x_k) - \nabla f(x^\star),{x_k} - {x^ \star } \right\rangle  \\
    & \hspace{0.4cm} + \gamma_k^2 \sigma_k  \alpha L (1+\theta)\mathbf{E}\langle \nabla f(x_k) - \nabla f(x^\star), x_k -x^\star\rangle  \\
    & \hspace{0.4cm} + \gamma_k^2 \sigma_k  \alpha (1+1/\theta)\sum_{i=1}^m\mathbf{E} \|\nabla f_i (x^\star)\|^2,
    \end{align*}
    where  the last inequality comes from \eqref{eqn:trick}, $\nabla f(x) = \sum_{i=1}^m \nabla f_i(x)$, and $\nabla f(x^\star) = 0$. Now, let $\gamma_k = 1/\left(L\alpha (1+\theta) \sigma_k\right).$ Then, by strong convexity of $f$, we have: 
    \begin{align*}
    \mathbf{E}\|x_{k+1} - x^\star \|^2 &  \leq \mathbf{E}\| x_k - x^\star \|^2 - \gamma_k \mathbf{E}\left\langle \nabla f(x_k) -\nabla f(x^\star),{x_k} - {x^ \star } \right\rangle  \\
    &\hspace{0.4cm}+ \gamma_k \frac{1}{\theta L} \sum_{i=1}^m \mathbf{E}\left\|   \nabla f_i(x^\star) \right\|^2 \\ 
    & \leq \rho_k \mathbf{E}\| x_k - x^\star \|^2 + \gamma_k \frac{1}{\theta L} \sum_{i=1}^m \mathbf{E}\left\|   \nabla f_i(x^\star) \right\|^2,
    \end{align*}
    where $\rho_k = 1-\mu\gamma_k$. 
    Define $\rho_{\max} \in (0,1)$ and $\gamma_{\max}$ such that $\rho_k \leq \rho_{\max}$ and $\gamma_k \leq  \gamma_{\max} , \forall k$. Then,
    \begin{align*}
    \Vert x_{k+1}-x^{\star}\Vert^2 &\leq \rho_{\max} \Vert x_k-x^{\star}\Vert^2 + e_k
    \end{align*}
    where
    \begin{align*}
    e_k &= \gamma_{\max} \frac{1}{\theta L} \sum_{i=1}^m \left\|   \nabla f_i(x^\star) \right\|^2.
    \end{align*}
    Consequently, $V_k \leq \rho_{\max}^k V_0 + \bar e$ where 
    $\bar e = e/(1-\rho_{\max})$. 
    
    If, instead, we use $\gamma_k = \gamma=1/(L\alpha(1+\theta)\sigma)$, then a similar argument yields that 
    \begin{align*}
    \mathbf{E}\Vert x_{k}-x^{\star}\Vert^2 &\leq (1-\mu\gamma)^k \Vert x_0-x^{\star}\Vert^2 + 
    \frac{1}{\mu \theta L }\sum_{i=1}^m \left\|   \nabla f_i(x^\star) \right\|^2.
    \end{align*}

 %=========================================================================================
    %
    % Proof theorem: distributed Q-GD for convex optimization 
    %
    %=========================================================================================
    \section{Proof of Theorem \ref{thm:DistributedQGDConvex}} \label{app:thm:DistributedQGDConvex}
    Following the proof in Theorem \ref{thm:DistributedQGDSC}, we reach: 
    
    \begin{align*}
        \mathbf{E}\|x_{k+1} - x^\star \|^2  & \leq \mathbf{E}\| x_k - x^\star \|^2 - 2\gamma \mathbf{E}\left\langle \nabla f(x_k) - \nabla f(x^\star),{x_k} - {x^ \star } \right\rangle  \\
        & \hspace{0.4cm} + \gamma^2 \sigma  \alpha L (1+\theta)\mathbf{E}\langle \nabla f(x_k) - \nabla f(x^\star), x_k -x^\star\rangle  \\
        & \hspace{0.4cm} + \gamma^2 \sigma  \alpha (1+1/\theta)\sum_{i=1}^m\mathbf{E} \|\nabla f_i (x^\star)\|^2.
    \end{align*}
     Now, let $\gamma = 1/\left(L\alpha (1+\theta) \sigma\right).$ Then, we have: 
    \begin{align*}
        \mathbf{E}\|x_{k+1} - x^\star \|^2 &  \leq \mathbf{E}\| x_k - x^\star \|^2 - \gamma \mathbf{E}\left\langle \nabla f(x_k) -\nabla f(x^\star),{x_k} - {x^ \star } \right\rangle \\
        &\hspace{0.4cm}+ \gamma \frac{1}{\theta L} \sum_{i=1}^m \mathbf{E}\left\|   \nabla f_i(x^\star) \right\|^2 \\ 
        & \leq \mathbf{E}\| x_k - x^\star \|^2 - \gamma \mathbf{E}(f(x_k) - f(x^\star)) + \gamma \frac{1}{\theta L} \sum_{i=1}^m \mathbf{E}\left\|   \nabla f_i(x^\star) \right\|^2,
    \end{align*}
    where the second inequality derives from the convexity of $f$, \emph{i.e.} $\left\langle \nabla f(x_k),{x_k} - {x^ \star } \right\rangle \geq f(x_k) - f(x^\star)$. Denote $\bar x_T = \frac{1}{T}\sum_{k=0}^{T-1} x_k$. Due to the convexity of the objective function $f$, $f(\bar x_T) \leq \frac{1}{T}\sum_{k=0}^{T-1} f(x_k)$. 
    By the manipulation, we have 
    \begin{align*}
       \mathbf{E}(f(x_T) - f(x^\star)) & \leq \frac{1}{T}\sum_{k=0}^{T-1} \mathbf{E}(f(x_k) - f(x^\star)) \\ 
       & \leq \frac{1}{T}\sum_{k=0}^{T-1} \frac{1}{\gamma} \left( \mathbf{E}\| x_k - x^\star \|^2 - \mathbf{E}\| x_{k+1} - x^\star \|^2 \right) \\
&\hspace{0.5cm}+ \frac{1}{\theta L}\sum_{i=1}^m\mathbf{E} \|\nabla f_i (x^\star)\|^2 \\ 
       & \leq \frac{1}{\gamma}\frac{1}{T}\|x_0 - x^\star \|^2 + \frac{1}{\theta L}\sum_{i=1}^m\mathbf{E} \|\nabla f_i (x^\star)\|^2,
    \end{align*}
    where we reach the last inequality by the telescopic series and by the non-negativity of the Euclidean norm.

       %===============================================================================================
    %
    % Proof of dist Q-IAG for strongly convex optimization 
    %
    %===============================================================================================
    
    \section{Proof of Theorem \ref{thm:QIAGDistSC}} \label{app:thm:QIAGDistSC}
    Denote $g_k = \sum_{i=1}^m \nabla f_i(x_{k-\tau_k^i})$. Before deriving the convergence rate, we introduce an essential lemma for our main analysis. 
    
    \begin{lemma}
    \label{lemma:trickerror_IAG}
  Let $\bar L =L\sqrt{m(1+\Delta)}$, and $\Delta = {\min}(\Delta_{\rm ave},\Delta_{\max})$. Consider the IAG update \eqref{eqn:QIAG} with the URQ according to Definition \ref{def:UnbiasedRandQuant}. Then,  
    \begin{align*}
        \mathbf{E}\left\| e_k \right\|^2 
        &\leq  2m\sigma\alpha  L^2 \gamma^2 \bar L^2\tau^2 \mathop{\max}\limits_{s \in [k-2\tau,k]} \left\|  x_{s} - x^\star \right\|^2  +  2m\alpha \gamma^2 \bar L^2\tau^2 \sum_{i=1}^m \left\|  \nabla f_i(x^\star) \right\|^2  
    \end{align*}
    where $  \sigma = \min\left( \sqrt{m(1+\Delta_{\rm ave})} , 1+ \Delta_{\max} \right) $, and $g_k = \sum_{i=1}^m \nabla f_i(x_{k-\tau_k^i})$.
    \end{lemma}
    \begin{proof}
    Denote $g_k = \sum_{i=1}^m \nabla f_i(x_{k-\tau_k^i})$ and $e_k = \nabla f(x_k) - g_k$.  Following the proof of Lemma \ref{lemma:LipschitzWholeFcn} with $x = x_k$ and $y=x_{k-\tau_k^i}$ yields 
    \begin{align*}
          \left\|  e_k \right\|^2 
        &  \leq \bar L^2 \mathop{\max}\limits_{i\in[1,m]} \|   x_k - x_{k-\tau_k^i} \|^2,
    \end{align*}
     where $\bar L =L\sqrt{m(1+\Delta)}$, and $\Delta = {\min}(\Delta_{\rm ave},\Delta_{\max})$. Next, notice that 
    \begin{align*}
        \left\|  e_k \right\|^2 
        & \leq \bar L^2 \mathop{\max}\limits_{i\in[1,m]}\tau_k^i \sum_{j = k-\tau_k^i}^{k-1}\left\| x_{j+1} - x_{j} \right\|^2 \\
        & \leq \bar L^2 \tau \sum_{j = k-\tau}^{k-1}\left\| x_{j+1} - x_{j} \right\|^2 \\ 
        & \leq \gamma^2 \bar L^2 \tau  \sum_{j = k-\tau}^{k-1}\left\| \sum_{i=1}^m Q(\nabla f_i(x_{j-\tau^i_j})) \right\|^2, 
    \end{align*}
    where the second inequality follows from the bounded delay assumption, and the last inequality from \eqref{eqn:QIAG}. On the other hand,
    \begin{align*}
        \mathbf{E}\left\| \sum_{i=1}^m Q(\nabla f_i(x_{j-\tau^i_j})) \right\|^2
        & \leq  \sigma \sum_{i=1}^m \mathbf{E}\left\| Q(\nabla f_i(x_{j-\tau^i_j})) \right\|^2  \\ 
        & \leq  \sigma \alpha\sum_{i=1}^m \left\| \nabla f_i(x_{j-\tau^i_j}) \right\|^2\\
        &  \leq 2\sigma \alpha \sum_{i=1}^m \left\|  \nabla f_i(x_{j-\tau^i_j}) - \nabla f_i(x^\star) \right\|^2 +   2\sigma\alpha \sum_{i=1}^m \left\|  \nabla f_i(x^\star) \right\|^2\\ 
        & \leq 2\sigma\alpha\sum_{i=1}^m L^2 \left\|  x_{j-\tau^i_j} - x^\star \right\| ^2+   2\sigma \alpha \sum_{i=1}^m \left\|  \nabla f_i(x^\star) \right\|^2 \\ 
        & \leq  2m\sigma\alpha  L^2  \mathop{\max}\limits_{s \in [j-\tau,j]} \left\|  x_{s} - x^\star \right\|^2 +   2m\alpha\sum_{i=1}^m \left\|  \nabla f_i(x^\star) \right\|^2,
    \end{align*}
    where we reach the first inequality by Lemma \ref{lemma:datasparsityTrick} due to Assumption \ref{assum:SparsityEmpiricalLoss}; the second inequality by the second property of Definition \ref{def:UnbiasedRandQuant}; the third inequality by $\|a+b\|^2\leq 2\|a\|^2+2\|b\|^2$; the forth inequality by the Lipschitz continuity assumption for gradient of each $f_i$; and the last inequality by the bounded delay assumption. Hence, plugging this result into the upper bound of $e_k$ yields the result. 
\iffalse
    \begin{align*}
        \mathbf{E}\left\| e_k \right\|^2 
        &\leq  2m\sigma\alpha  L^2 \gamma^2 \bar L^2\tau^2 \mathop{\max}\limits_{s \in [k-2\tau,k]} \left\|  x_{s} - x^\star \right\|^2  +  2m\alpha \gamma^2 \bar L^2\tau^2 \sum_{i=1}^m \left\|  \nabla f_i(x^\star) \right\|^2.
    \end{align*}
\fi
    \end{proof}
    We now prove Theorem \ref{thm:QIAGDistSC}. From \eqref{eqn:QIAG}, we have
    \begin{align*}
        \left\| x_{k+1} - x^\star \right\|^2 & = \left\| x_k - x^\star \right\|^2 - 2\gamma\left\langle  \sum_{i=1}^m Q\left( \nabla f_i (x_{k-\tau_k^i}) \right), x_k -x^\star \right\rangle\\
        &\hspace{0.4cm}+ \gamma^2 \left\| \sum_{i=1}^m Q\left( \nabla f_i (x_{k-\tau_k^i}) \right)  \right\|^2.
    \end{align*}
    Taking the expectation over all the random variables yields 
    \begin{align*}
          \mathbf{E}\left\| x_{k+1} - x^\star \right\|^2 & = \mathbf{E}\left\| x_k - x^\star \right\|^2 - 2\gamma\mathbf{E}\left\langle  \sum_{i=1}^m  \nabla f_i (x_{k-\tau_k^i}) , x_k -x^\star \right\rangle \\
          &\hspace{0.4cm}+ \gamma^2 \mathbf{E}\left\| \sum_{i=1}^m Q\left( \nabla f_i (x_{k-\tau_k^i}) \right)  \right\|^2.
    \end{align*}
    Using the second property in Definition \ref{def:UnbiasedRandQuant} and Lemma \ref{lemma:datasparsityTrick} due to Assumption \ref{assum:SparsityEmpiricalLoss}, we get 
    \begin{align*}
           \mathbf{E}\left\| x_{k+1} - x^\star \right\|^2 
          % & \leq \mathbf{E}\left\| x_k - x^\star \right\|^2 - 2\gamma\mathbf{E}\left\langle  \sum_{i=1}^m  \nabla f_i (x_{k-\tau_k^i}) , x_k -x^\star \right\rangle \\ 
          % &\hspace{0.4cm} + \gamma^2 \sigma\alpha \sum_{i=1}^m  \mathbf{E} \left\|   \nabla f_i (x_{k-\tau_k^i})   \right\|^2 \\ 
           & \leq  \mathbf{E}\left\| x_k - x^\star \right\|^2  - 2\gamma \mathbf{E}\left\langle  \nabla f(x_k) , x_k -x^\star \right\rangle \\ 
           & \hspace{0.4cm} +  2\gamma \mathbf{E}\left\langle  g_k - \nabla f(x_k) , x_k -x^\star \right\rangle \\
           &\hspace{0.4cm} + \gamma^2 \sigma \alpha \sum_{i=1}^m  \mathbf{E} \left\|   \nabla f_i (x_{k-\tau_k^i})   \right\|^2 \\ 
           & \leq \mathbf{E}\left\| x_k - x^\star \right\|^2  - 2\gamma \mathbf{E}\left\langle  \nabla f(x_k) , x_k -x^\star \right\rangle \\ 
           & \hspace{0.4cm} +  \mathbf{E}\left\| g_k - \nabla f(x_k) \right\|^2+ \gamma^2\mathbf{E} \left\| x_k -x^\star \right\|^2  \\
           &\hspace{0.4cm} + (1+\theta)\gamma^2 \sigma \alpha \sum_{i=1}^m  \mathbf{E} \left\|   \nabla f_i (x_{k-\tau_k^i}) -    \nabla f_i (x^\star) \right\|^2 \\
           & \hspace{0.4cm} + (1+1/\theta)\gamma^2 \sigma  \alpha \sum_{i=1}^m  \mathbf{E} \left\|     \nabla f_i (x^\star) \right\|^2 
    \end{align*}
    where $\sigma = \min\left( \sqrt{m(1+\Delta_{\rm ave})},1+\Delta_{\max} \right)$ $g_k = \sum_{i=1}^m \nabla f_i(x_{k-\tau_k^i})$. The second inequality follows from Cauchy-Schwarz's inequality and from the fact that $\| x + y \|^2\leq (1+\theta)\| x\|^2  + (1+1/\theta)\|y \|^2$ for $x,y\in\mathbb{R}^d$ and $\theta>0$.  Due to the Lipschitz continuity assumption of $\nabla f_i$, we get    
    \begin{align*}
      \mathbf{E}\left\| x_{k+1} - x^\star \right\|^2      & \leq (1+\gamma^2)\mathbf{E}\left\| x_k - x^\star \right\|^2  - 2\gamma \mathbf{E}\left\langle  \nabla f(x_k) , x_k -x^\star \right\rangle \\ 
           & \hspace{0.4cm} +  \mathbf{E}\left\| g_k - \nabla f(x_k) \right\|^2 \ \\
           &\hspace{0.4cm} + (1+\theta)\gamma^2 m\alpha \sigma  L^2 \mathbf{E} \left\|  x_{k-\tau_k^i} -    x^\star \right\|^2 \\
           & \hspace{0.4cm} + (1+1/\theta)\gamma^2 \sigma  \alpha \sum_{i=1}^m   \mathbf{E} \left\|     \nabla f_i (x^\star) \right\|^2.
    \end{align*}
 It follows from Lemma \ref{lemma:trickerror_IAG} that 
    \begin{align*}
           \mathbf{E}\left\| x_{k+1} - x^\star \right\|^2 
           & \leq (1+\gamma^2)\mathbf{E}\left\| x_k - x^\star \right\|^2  - 2\gamma \mathbf{E}\left\langle  \nabla f(x_k) , x_k -x^\star \right\rangle \\ 
        &\hspace{0.4cm} +  2m\sigma\alpha  L^2 \gamma^2 \bar L^2\tau^2 \mathop{\max}\limits_{s \in [k-2\tau,k]} \left\|  x_{s} - x^\star \right\|^2  \\
       & \hspace{0.4cm}+  2m\alpha \gamma^2 \bar L^2\tau^2 \sum_{i=1}^m \left\|  \nabla f_i(x^\star) \right\|^2\\
                      &\hspace{0.4cm} + (1+\theta)\gamma^2 m\alpha\sigma L^2 \mathop{\max}\limits_{s\in[k-\tau,k]} \mathbf{E} \left\|  x_{s} -    x^\star \right\|^2 \\
           & \hspace{0.4cm} + (1+1/\theta)\gamma^2 \sigma  \alpha \sum_{i=1}^m   \mathbf{E} \left\|     \nabla f_i (x^\star) \right\|^2. 
    \end{align*}
   Due to the property of the strong convexity assumption of $f$, it holds for $x,y\in\mathbb{R}^d$ that
    \begin{align*}
        \langle \nabla f(x) - \nabla f(y), x - y \rangle \geq \mu \| x -y \|^2,  
    \end{align*}
    Using this inequality with $x= x_k, y = x^\star$ and notice that $\nabla f(x^\star) = 0$ yields
    \begin{align*}
        V_{k+1} & \leq p V_k + q \mathop{\max}\limits_{s \in[k-2\tau,k]} V_s + e, \\ \intertext{where}
        V_k & = \mathbf{E}\|x_k - x^\star \|^2 \\ 
        p & = 1 - 2\mu\gamma + \gamma^2   \\ 
        q & = 2m\sigma\alpha  L^2 \gamma^2 \bar L^2\tau^2 +  (1+\theta)\gamma^2 m\alpha\sigma L^2\\ 
        e & = \left(2m\alpha \gamma^2 \bar L^2\tau^2 + (1+1/\theta)\gamma^2 \sigma  \alpha\right) \sum_{i=1}^m   \mathbf{E} \left\|     \nabla f_i (x^\star) \right\|^2.
    \end{align*}
 From Lemma 1 of \cite{aytekin2014asynchronous}, $p+q < 1$ implies that 
    \begin{align*}
        \gamma < \frac{2\mu}{1 + m\sigma\alpha L^2\left( 2\bar L^2\tau^2 + (1+\theta)  \right)}.
    \end{align*}
    Then, this implies that $V_k \leq (p+q)^{k/(1+2\tau)} V_0 + e/(1-p-q).$

    \section{Proof of Theorem \ref{thm:DistIAGHamidNSC}}\label{app:thm:DistIAGHamidNSC}

     Denote $\tilde g_k = \sum_{i=1}^m Q\left( \nabla f_i(x_{k-\tau_k^i})\right)$ and $g_k = \sum_{i=1}^m \nabla f_i(x_{k-\tau_k^i})$. Before deriving the convergence rate, we introduce the lemmas which are instrumental in our main analysis.
    \begin{lemma}
    \label{lemma:QIAGNSC1}
    The sequence $\{x_k \}$ generated by \eqref{eqn:QIAG} satisfies 
    \begin{align*}
       \|  g_k - \nabla f(x_k)\|^2 \leq \bar L^2 \gamma^2\tau \sum_{j=k-\tau}^{k-1} \| \tilde g_j \|^2. 
    \end{align*}
    \end{lemma}
    \begin{proof}
    Following the proof in Lemma  \ref{lemma:errorQIAGnondist1} yields the result. 
    \end{proof}
    
    \begin{lemma}
    \label{lemma:QIAGNSC1.1}
     The sequence $\{x_k \}$ generated by \eqref{eqn:QIAG} under Assumption \ref{assum:SparsityEmpiricalLoss} satisfies 
     \begin{align*}
         \left\| g_k - \tilde g_k \right\|^2 \leq \sigma \sum_{i=1}^m \left\| \nabla f_i(x_{k-\tau_k^i}) - Q\left( \nabla f_i(x_{k-\tau_k^i}) \right) \right\|^2,
     \end{align*}
     where $\sigma = {\min}\left( \sqrt{m(1+\Delta_{\rm ave})},1+\Delta_{\max} \right).$
    \end{lemma}
    \begin{proof}
    The proof arguments follow those in Lemma \ref{lemma:datasparsityTrick} with replacing $Q(\nabla f_i(x_k))$ with $\nabla f_i(x_{k-\tau_k^i}) - Q\left( \nabla f_i(x_{k-\tau_k^i})\right)$. Also, note that \[{\rm supp}\left( \nabla f_i(x_{k-\tau_k^i}) - Q\left( \nabla f_i(x_{k-\tau_k^i})\right) \right)\subset {\rm supp}(\nabla f_i(x_{k-\tau_k^i})),\] and Assumption \ref{assum:SparsityEmpiricalLoss} implies that ${\rm supp}(\nabla f_i(x_{k-\tau_k^i}))$ can be computed from the data directly. 
    \end{proof}

    \begin{lemma}
    \label{lemma:QIAGNSC2}
    The sequence $\{x_k \}$ generated by \eqref{eqn:QIAG} under Assumption \ref{assum:SparsityEmpiricalLoss} and  \ref{assum:BoundedGradient} satisfies 
    \begin{align*}
        \mathbf{E} \| \tilde g_k - \nabla f(x_k)\|^2 \leq   2 \bar L^2 \gamma^2\tau \sum_{j=k-\tau}^{k-1} \mathbf{E} \| \tilde g_j \|^2 + 2 \beta\sigma m C^2.
    \end{align*}
    \end{lemma}
    \begin{proof}
    By the fact that $\|x+y\|^2\leq 2\|x\|^2+2\|y\|^2$, we have: 
    \begin{align*} 
    \| \tilde g_k - \nabla f(x_k)\|^2 
    & \leq  2  \|  g_k - \nabla f(x_k)\|^2 + 2  \|  g_k - \tilde g_k\|^2 \\ 
    & \leq 2 \bar L^2 \gamma^2\tau \sum_{j=k-\tau}^{k-1} \| \tilde g_j \|^2 + 2  \|  g_k - \tilde g_k\|^2,
    \end{align*}
    where the last inequality follows from Lemma \ref{lemma:QIAGNSC1}. Next, taking the expectation of the inequality from Lemma  \ref{lemma:QIAGNSC1.1} over the randomness yields 
    \begin{align*}
         \mathbf{E}\|  g_k - \tilde g_k\|^2  & \leq \sigma  \sum_{i=1}^m  \mathbf{E}\left\|  \nabla f_i(x_{k-\tau_k^i}) - Q\left( \nabla f_i(x_{k-\tau_k^i}) \right) \right\|^2 \\ 
                                                              & \leq  \beta\sigma  \sum_{i=1}^m  \mathbf{E}\left\|  \nabla f_i(x_{k-\tau_k^i})  \right\|^2,
    \end{align*}
where we reach the last inequality by the second property of the URQ, \emph{i.e.} $\mathbf{E}\|Q(v)-v\|^2\leq \beta\mathbf{E}\|v\|^2$. Next, taking the expectation over the randomness yields 
    \begin{align*}
         \mathbf{E}\| \tilde g_k - \nabla f(x_k)\|^2 
    & \leq 2 \bar L^2 \gamma^2\tau \sum_{j=k-\tau}^{k-1} \mathbf{E} \| \tilde g_j \|^2 + 2  \mathbf{E}\|  g_k - \tilde g_k\|^2 \\ 
    & \leq  2 \bar L^2 \gamma^2\tau \sum_{j=k-\tau}^{k-1} \mathbf{E} \| \tilde g_j \|^2 + 2 \beta\sigma   \sum_{i=1}^m  \mathbf{E}\left\|  \nabla f_i(x_{k-\tau_k^i})  \right\|^2 \\ 
    & \leq 2 \bar L^2 \gamma^2\tau \sum_{j=k-\tau}^{k-1} \mathbf{E} \| \tilde g_j \|^2 + 2 \beta\sigma m C^2,
    \end{align*}
    where the last inequality results from Assumption \ref{assum:BoundedGradient}. 
    \end{proof}
    
    \begin{lemma}
    \label{lemma:trickMAinQIAGNSCWithErr}
     Assume that non-negative sequences $\{V_k\}, \{w_k\},$ and $\{\Theta_k\}$ satisfying the following inequality 
\begin{equation}
      V_{k+1} \leq V_k - a\Theta_k - b w_{k} + c \sum_{j=k-\tau}^k w_{j} + e, 
\end{equation}
where $a,b,c,e >0$. Further suppose that $b-c(\tau + 1) \geq 0$ and $w_k =0 $  for $k < 0$. Then,  
\begin{align*}
     \frac{1}{K+1}\sum_{k=0}^{K}\Theta_k 
    & \leq \frac{1}{a}\frac{1}{K+1}( V_0 - V_{K+1} ) +\frac{1}{a} e.
\end{align*}
 
    \end{lemma}
    \begin{proof}
    Following the proof in Lemma \ref{lemma:IAGConvegenceTrickResult} yields the result.  
    \end{proof}

    Now, we are ready to derive the convergence rate. From the Lipschitz continuity assumption of the gradient of $f$ and from the fact that $2\langle x,y \rangle= \left\|x\right\|^2+\left\|y\right\|^2-\left\|x-y\right\|^2$ for $x,y\in\mathbb{R}^d$, 
    \begin{align*}
        f(x_{k+1}) - f^\star
%& \leq f(x_k) - f^\star - \gamma \left\langle \nabla f(x_k), \tilde g_k \right\rangle + \frac{\bar L\gamma^2}{2}\left\| \tilde g_k \right\|^2  \\ 
        & \leq f(x_k) - f^\star - \frac{\gamma}{2}\left\| \nabla f(x_k) \right\|^2 - \left(\frac{\gamma}{2} - \frac{\bar L\gamma^2}{2} \right) \left\| \tilde g_k \right\|^2 \\
&\hspace{0.5cm}  + \frac{\gamma}{2}\left\| \tilde g_k -\nabla f(x_k) \right\|^2 ,
    \end{align*}
   Taking the expectation over the randomness and using Lemma \ref{lemma:QIAGNSC2} yields 
    \begin{align*}
      \mathbf{E} f(x_{k+1}) - f^\star 
%& \leq \mathbf{E} f(x_k) - f^\star - \frac{\gamma}{2}\mathbf{E}\left\| \nabla f(x_k) \right\|^2 - \left(\frac{\gamma}{2} - \frac{\bar L\gamma^2}{2} \right) \mathbf{E}\left\| \tilde g_k \right\|^2  + \frac{\gamma}{2}\mathbf{E}\left\| \tilde g_k -\nabla f(x_k) \right\|^2 \\ 
      & \leq \mathbf{E} f(x_k) - f^\star - \frac{\gamma}{2}\mathbf{E}\left\| \nabla f(x_k) \right\|^2 - \left(\frac{\gamma}{2} - \frac{\bar L\gamma^2}{2} \right)  \mathbf{E}\left\| \tilde g_k \right\|^2  \\
      &\hspace{0.4cm}+ \bar L^2 \gamma^3\tau \sum_{j=k-\tau}^{k-1} \mathbf{E} \| \tilde g_j \|^2 + \gamma \beta\sigma m C^2
    \end{align*}
    Next, applying Lemma \eqref{lemma:trickMAinQIAGNSCWithErr} with $V_k = \mathbf{E} f(x_k) - f^\star, \Theta_k =\mathbf{E}\| \nabla f(x_k)\|^2, w_k = \mathbf{E}\|\tilde g_k\|^2, e = \gamma\beta\sigma m C^2, a=\gamma/2, b= \gamma/2-\bar L\gamma^2/2,$ and $c = \bar L^2\tau\gamma^3$ yields the result. 
    \begin{align*}
           \mathop{\min}\limits_{k\in[0,K]} \mathbf{E}\| \nabla f(x_k) \|^2 \leq \frac{1}{a}\frac{1}{K+1}\left( f(x_0) - f^\star\right)- \frac{1}{a}\frac{1}{K+1}\left(f(x_k) - f^\star \right) + \frac{1}{a}e.
    \end{align*}
    Note that $w_k =0$ for $k< 0$ since $\mathbf{E}\| \tilde g_k \|^2= \mathbf{E}\| x_{k+1}-x_k\|^2/\gamma^2.$  Lastly, we need to find the admissible range of the step-size which guarantees the convergence. The following criteria must be satisfied: $b>0$ and $b-c(\tau+1)\geq 0$. The first criterion implies that $\gamma <1/\bar L$. The second criterion implies that
    \begin{align*}
     \frac{\gamma}{2} - \frac{\bar L \gamma^2}{2} - \bar L^2 \tau(\tau+1) \gamma^3 \geq 0. 
    \end{align*}
    Lastly, let $\gamma = 1/(\bar L +\omega)$ for $\omega >0$ and plugging the expression into the result yields 
    \begin{align*}
        \omega^2 + \bar L \omega - 2 \bar L^2 \tau(\tau+1) \geq 0,
    \end{align*}
    and therefore 
    \begin{align*}
        \omega \geq \left(-1 + \sqrt{1+8\tau(\tau+1)}\right)\frac{\bar L}{2}. 
    \end{align*}
    Thus, we can conclude that the admissible range of the step-size is 
    \begin{align*}
        \gamma < \frac{1}{1+\sqrt{1+8\tau(\tau+1)}}\frac{2}{\bar L}.
    \end{align*}

\section*{Acknowledgments}
   This work was partially supported by the Wallenberg AI, Autonomous Systems and Software Program (WASP) funded by the Knut and Alice Wallenberg Foundation.

%\bibliographystyle{siamplain}
%\bibliography{sample2}

\end{document}